\documentclass[reqno,11pt]{amsart}
\usepackage{amsmath,amsfonts,amssymb,amsxtra,latexsym,amscd,enumerate,amsthm,verbatim}

\usepackage{graphicx}

\usepackage[dvipsnames]{xcolor}
\usepackage{hyperref}
\hypersetup{colorlinks=true, pdfstartview=FitV, linkcolor=BrickRed,citecolor=BrickRed, urlcolor=BrickRed}

\usepackage{cancel,soul}

\usepackage[margin=1.40in]{geometry}
\numberwithin{equation}{section}

\newcommand{\R}{\mathbb{R}}

\newcommand{\T}{\mathbb{T}}
\newcommand{\C}{\mathbb{C}}
\newcommand{\Z}{\mathbb{Z}}
\newcommand{\eps}{\epsilon}

\numberwithin{equation}{section} %pour numeroter les equations par section

\newtheorem{theorem}{Theorem}[section]
\newtheorem{lemma}[theorem]{Lemma}
\newtheorem{proposition}[theorem]{Proposition}

\newtheorem{remark}[theorem]{Remark}

\begin{document}

\title{Nonlinear Landau damping and wave operators in sharp Gevrey spaces}

\author{Alexandru D.\ Ionescu}
\address{Princeton University}
\email{aionescu@math.princeton.edu}
\author{Benoit Pausader}
\address{Brown University}
\email{benoit\_pausader@brown.edu}
\author{Xuecheng Wang}
\address{YMSC, Tsinghua University \& BIMSA}
\email{xuecheng@tsinghua.edu.cn}
\author{Klaus Widmayer}
\address{University of Zurich and University of Vienna}
\email{klaus.widmayer@math.uzh.ch}

\thanks{A.I.\ was supported in part by NSF grant DMS-2007008; B.P.\ was supported in part by NSF grant DMS-2154162; X.W.\  was supported in part by NSFC-12141102, and MOST-2020YFA0713003. K.W.\ gratefully acknowledges support of the SNSF through grant PCEFP2\_203059.}
\maketitle

\begin{abstract}
We prove nonlinear Landau damping in optimal weighted Gevrey-3 spaces for solutions of the confined Vlasov-Poisson system on $\T^d\times\R^d$ which are small perturbations of homogeneous Penrose-stable equilibria.

We also prove the existence of nonlinear scattering operators associated to the confined Vlasov-Poisson evolution, as well as suitable injectivity properties and Lipschitz estimates (also in weighted Gevrey-3 spaces) on these operators.

Our results give definitive answers to two well-known open problems in the field, both of them stated in the recent review of Bedrossian \cite[Section 6]{Be}. 
\end{abstract}

\maketitle

\setcounter{tocdepth}{1}

\tableofcontents

\section{Introduction} In this article we investigate asymptotic properties of solutions of the classical Vlasov-Poisson system. In the confined case the electron distribution is modeled by a function $F=F(t,x,v):I\times\mathbb{T}^d_x\times\mathbb{R}^d_v\to\R$, where $I\subseteq\R$ is an interval, $\T:=\R/(2\pi\Z)$, and $d\geq 1$, evolving according to the nondimensionalized Vlasov-Poisson system
\begin{equation}\label{VPO}
\begin{split}
\left(\partial_t+v\cdot\nabla_x\right)F+\nabla_x\phi\cdot\nabla_vF=0,\qquad \Delta_x\phi=\rho=\int_{\mathbb{R}^d}Fdv-1.
\end{split}
\end{equation}
Our goal is to investigate the global nonlinear stability properties (Landau damping) of solutions around a {\it{spatially homogeneous equilibrium}} $M_0=M_0(v)$ normalized such that 
\begin{equation}\label{VP00}
\int_{\R^d}M_0(v)\,dv =1.
\end{equation}
Looking for solutions of \eqref{VPO} of the form $F=M_0+f$, we obtain the system
\begin{equation}\label{NVP}
\begin{split}
&\left(\partial_t+v\cdot\nabla_x\right)f+E\cdot \nabla_vf+E\cdot\nabla_vM_0=0,\\
&E:=\nabla_x\Delta_x^{-1}\rho,\qquad\rho(t,x):=\int_{\R^d} f(t,x,v)\,dv.
\end{split}
\end{equation}

\subsection{The main theorems} The free transport equation $\partial_t f+v\cdot\nabla_x f=0$ exhibits phase mixing, which leads to time decay of the spatial density $\rho$. It was a fundamental observation of Landau \cite{Lan1946} (see also \cite[Chapter 3]{LLX1981}) that an interesting mechanism of decay exists also in the linearized Vlasov-Poisson equations near homogeneous equilibria satisfying certain conditions (nowadays called "Penrose criterion'' \cite{Pe}, see \eqref{eq:Penrose} below).   

The classical mechanism for nonlinear stability involves trading regularity for decay. In order to state our main theorems we need some definitions. For any $\lambda,s\in(0,1]$ we define the Gevrey spaces $\mathcal{G}^{\lambda,s}(\T^d\times\R^d)$ induced by the norms
\begin{equation}\label{GevNorms}
\|g\|_{\mathcal{G}^{\lambda,s}(\T^d\times\R^d)}:=\big\|\widehat{g}(k,\xi)e^{\lambda\langle k,\xi\rangle^s}\big\|_{L^2_{k,\xi}},
\end{equation}
where $\langle k,\xi\rangle:=\sqrt{1+|k|^2+|\xi|^2}$ and $\widehat{g}:\Z^d\times\R^d\to\C$ denotes the Fourier transform of the function $g$,
\begin{equation*}
\widehat{g}(k,\xi):=\int_{\T^d\times\R^d}g(x,v)e^{-ik\cdot x}e^{-i\xi\cdot v}\,dxdv.
\end{equation*}
For $n\in\Z_+$ we define also the weighted Gevrey spaces $\mathcal{G}^{\lambda,s}_n(\T^d\times\R^d)$ induced by the norms
\begin{equation}\label{GevNorms2}
\|g(x,v)\|_{\mathcal{G}^{\lambda,s}_n(\T^d\times\R^d)}:=\sum_{|a|\leq n}\|D^a_\xi\widehat{g}(k,\xi)e^{\lambda\langle k,\xi\rangle^s}\|_{\mathcal{G}^{\lambda,s}(\T^d\times\R^d)}.
\end{equation}
We also define the spaces $\mathcal{G}^{\lambda,s}(\T^d)$, $\mathcal{G}^{\lambda,s}(\R^d)$, and $\mathcal{G}^{\lambda,s}_n(\R^d)$ in a similar fashion. For simplicity of notation, we sometimes let $\mathcal{G}^{\lambda,s}_n:=\mathcal{G}^{\lambda,s}_n(\T^d\times\R^d)$.

We assume that there are constants $\lambda_0,\vartheta\in(0,1]$ such that the homogenous equilibrium $M_0$ satisfies the smoothness and decay bounds
\begin{equation}\label{M01}
\sup_{\xi\in\R^d}e^{\lambda_0\langle\xi\rangle^{1/3}}\big\{\big|\widehat{M_0}(\xi)\big|+|\xi|\,\big|\nabla_{\xi}\widehat{M}_0(\xi)\big|\big\}+\|\nabla_v M_0\|_{\mathcal{G}^{\lambda_0,1/3}_{d'}(\R^d)}\leq \vartheta^{-1},
\end{equation}
where $d':=\lfloor d/2+1\rfloor$, and the \textit{Penrose criterion}
\begin{equation}\label{eq:Penrose}
\inf_{k\in\Z^d\setminus\{0\},\,\tau\in\C,\,\Im(\tau)\leq 0}\Big|1+\int_0^\infty e^{-i\tau s}s\widehat{M_0}(sk)\,ds\Big|\geq\vartheta>0.
\end{equation}

To state our main results it is convenient to define the kinetic profile 
\begin{equation}\label{NVP2}
g(t,x,v):=f(t,x+tv,v).
\end{equation}
In terms of the profile $g$, the Vlasov-Poisson system \eqref{NVP} becomes
\begin{equation}\label{NVP3}
\begin{split}
&\partial_tg(t,x,v)+E(t,x+tv)\cdot\nabla_vM_0(v)+E(t,x+tv)\cdot(\nabla_v-t\nabla_x)g(t,x,v)=0,\\
&E:=\nabla_x\Delta_x^{-1}\rho,\qquad\rho(t,x):=\int_{\R^d} g(t,x-tv,v)\,dv.
\end{split}
\end{equation}

Our first main theorem concerns the global nonlinear stability of solutions of the Vlasov-Poisson system \eqref{NVP}.

\begin{theorem}\label{MainThm}
(Nonlinear Landau damping) Assume $d\geq 1$, $\lambda_0,\vartheta\in(0,1]$, and $M_0$ is a homogeneous equilibrium satisfying the bounds \eqref{M01}--\eqref{eq:Penrose}. 

(i) There is a constant $\overline{\kappa}=\overline{\kappa}(d,\lambda_0,\vartheta)>0$ with the property that if $f_0$ is a Gevrey-smooth, small, and neutral perturbation satisfying
\begin{equation}\label{MainTHM1}
\|f_0\|_{\mathcal{G}_{d'}^{\lambda_0,1/3}(\T^d\times\R^d)}\leq\kappa_0\leq\overline{\kappa},\qquad \int_{\T^d\times\R^d} f_0(x,v)\,dxdv=0
\end{equation}
then there is a unique global solution $g\in C([0,\infty):\mathcal{G}_{d'}^{3\lambda_0/4,1/3})$ of the equation \eqref{NVP} with initial data $f(0)=f_0$. Moreover $\int_{\T^d\times\R^d} f(t,x,v)\,dxdv=0$ for any $t\in[0,\infty)$, and there is a function $g_\infty\in\mathcal{G}_{d'}^{3\lambda_0/4,1/3}$ (the final state) such that
\begin{equation}\label{MainTHM2}
\|\rho(t)\|_{\mathcal{G}^{\lambda_0/2,1/3}(\T^d)}+\|g(t)-g_\infty\|_{\mathcal{G}^{\lambda_0/2,1/3}_{d'}(\T^d\times\R^d)}\leq A\kappa_0e^{-\lambda_0\langle t\rangle^{1/3}/4}
\end{equation}
for any $t\in[0,\infty)$, where $A=A(d,\lambda_0,\vartheta)\in[1,\overline{\kappa}^{-1}]$ is a constant, $\rho$ denotes the macroscopic density defined in \eqref{NVP}, and $g$ denotes the profile defined in \eqref{NVP2}. 

(ii) The solution map $f_0\to g$ is a continuous map from $B_{\overline{\kappa}}(\mathcal{G}_{d'}^{\lambda_0,1/3})$ to $C\big([0,\infty):\mathcal{G}_{d'}^{3\lambda_0/4,1/3}\big)$. In fact, for any $f_{01},f_{02}\in B_{\overline{\kappa}}(\mathcal{G}_{d'}^{\lambda_0,1/3})$ and $t\in[0,\infty)$ we have the Lipschitz estimates
\begin{equation}\label{MainTHM4}
A^{-1}\|f_{01}-f_{02}\|_{\mathcal{G}_{d'}^{\lambda_0/2,1/3}}\leq \|(g_1-g_2)(t)\|_{\mathcal{G}_{d'}^{3\lambda_0/4,1/3}}\leq A\|f_{01}-f_{02}\|_{\mathcal{G}_{d'}^{\lambda_0,1/3}}.
\end{equation}
In particular, the final state operator $\mathcal{S}_+(f_0):=g_\infty$ is a Lipschitz continuous and injective map from $B_{\overline{\kappa}}(\mathcal{G}_{d'}^{\lambda_0,1/3})$ to $\mathcal{G}_{d'}^{3\lambda_0/4,1/3}$. 
\end{theorem}

Our second main theorem concerns the solvability of the final state problem, and the construction of injective wave operators satisfying Lipschitz bounds.

\begin{theorem}\label{MainThm22}
(The final state problem) Assume $d$, $\lambda_0$, $\vartheta$, and $M_0$ are as in Theorem \ref{MainThm}. Then there is a constant $\overline{\kappa}'=\overline{\kappa}'(d,\lambda_0,\vartheta)>0$ with the property that if $g_\infty$ satisfies
\begin{equation}\label{MainTHMXY1}
\|g_\infty\|_{\mathcal{G}_{d'}^{\lambda_0,1/3}}\leq\overline{\kappa}',\qquad \int_{\T^d\times\R^d} g_\infty(x,v)\,dxdv=0
\end{equation}
then there is a unique solution $g\in C([0,\infty):\mathcal{G}_{d'}^{3\lambda_0/4,1/3})$ of the system \eqref{NVP3} satisfying 
\begin{equation}\label{MainTHNXY2}
\lim_{t\to\infty}\|g(t)-g_\infty\|_{\mathcal{G}_{d'}^{3\lambda_0/4,1/3}}=0,\qquad \int_{\T^d\times\R^d} g(t,x,v)\,dxdv=0.
\end{equation}
Moreover, for any $g_{\infty1},g_{\infty2}\in B_{\overline{\kappa}'}(\mathcal{G}_{d'}^{\lambda_0,1/3})$ and $t\in[0,\infty)$ we have the bi-Lipschitz estimates
\begin{equation}\label{MainTHMXY4}
(A')^{-1}\|g_{\infty1}-g_{\infty2}\|_{\mathcal{G}_{d'}^{\lambda_0/2,1/3}}\leq \|(g_1-g_2)(t)\|_{\mathcal{G}_{d'}^{3\lambda_0/4,1/3}}\leq A'\|g_{\infty1}-g_{\infty2}\|_{\mathcal{G}_{d'}^{\lambda_0,1/3}},
\end{equation}
for some constant $A'=A'(d,\lambda_0,\vartheta)\in[1,1/\overline{\kappa}']$. In particular, the wave operator $\mathcal{W}_+(g_\infty):=g(0)$ is a Lipschitz continuous and injective map from $B_{\overline{\kappa}'}(\mathcal{G}_{d'}^{\lambda_0,1/3})$ to $\mathcal{G}_{d'}^{3\lambda_0/4,1/3}$. 
\end{theorem}

Finally, we can combine Theorems \ref{MainThm} and \ref{MainThm22} to prove the existence of nonlinear scattering operators associated to the confined Vlasov-Poisson system. More precisely:

\begin{theorem}\label{MainThm33}
(The scattering operator) Assume $d$, $\lambda_0$, $\vartheta$, and $M_0$ are as in Theorem \ref{MainThm}. Then there is a constant $\overline{\kappa}''=\overline{\kappa}''(d,\lambda_0,\vartheta)>0$ with the property that if $g_{-\infty}$ satisfies
\begin{equation}\label{MainTH1}
\|g_{-\infty}\|_{\mathcal{G}_{d'}^{\lambda_0,1/3}}\leq\overline{\kappa}'',\qquad \int_{\T^d\times\R^d} g_{-\infty}(x,v)\,dxdv=0
\end{equation}
then there is a unique neutral solution $g\in C(\R:\mathcal{G}_{d'}^{\lambda_0/2,1/3})$ of the system \eqref{NVP3} satisfying 
\begin{equation}\label{MainTH2}
\lim_{t\to-\infty}g(t)=g_{-\infty}\text{ in }\mathcal{G}_{d'}^{\lambda_0/2,1/3}\qquad\text{ and }\qquad\lim_{t\to\infty}g(t)=g_{\infty}\text{ in }\mathcal{G}_{d'}^{\lambda_0/2,1/3}.
\end{equation}
Moreover the scattering operator $\mathcal{S}(g_{-\infty}):=g_\infty$ is a Lipschitz continuous and injective map from $B_{\overline{\kappa}''}(\mathcal{G}_{d'}^{\lambda_0,1/3})$ to $\mathcal{G}_{d'}^{\lambda_0/2,1/3}$. In fact, for any $g_{-\infty1},g_{-\infty2}\in B_{\overline{\kappa}''}(\mathcal{G}_{d'}^{\lambda_0,1/3})$ we have
\begin{equation}\label{MainTH4}
(A'')^{-1}\|g_{-\infty1}-g_{-\infty2}\|_{\mathcal{G}_{d'}^{\lambda_0/4,1/3}}\leq \|g_{\infty1}-g_{\infty2}\|_{\mathcal{G}_{d'}^{\lambda_0/2,1/3}}\leq A''\|g_{-\infty1}-g_{-\infty2}\|_{\mathcal{G}_{d'}^{\lambda_0,1/3}},
\end{equation}
for some constant $A''=A''(d,\lambda_0,\vartheta)\in[1,1/\overline{\kappa}'']$. 
\end{theorem}

We conclude this subsection with some remarks:
\smallskip

\begin{remark} Our main theorems provide definitive answers to two of the open problems stated in the recent review article of Bedrossian \cite[Section 6]{Be}: 

(a) nonlinear Landau damping for initial perturbations in sharp weighted Gevrey-3 spaces, which was already conjectured by Mouhot-Villani \cite[Section 13]{MouVil};

(b) injectivity and Lipschitz bounds on the scattering operator $g_{-\infty}\mapsto g_\infty$ (and similar properties for the final state operator $f_0\mapsto g_\infty$ and the wave operator $g_\infty\mapsto g(0)$).

\end{remark}

\begin{remark} We also allow a larger class of equilibria $M_0$ than in the earlier work, in the same weighted Gevrey-3 space as the initial perturbation. For comparison, all the earlier proofs of nonlinear Landau damping in the confined case (in \cite{MouVil,BMM2016,GrNgRo}) require analytic equilibria. Our assumptions \eqref{M01} also allow certain slowly decreasing equilibria, like the Poisson equilibrium corresponding to $\widehat{M_0}(\xi)=e^{-|\xi|}$ in all dimensions $d\geq 1$, which were not included in the earlier nonlinear theorems. The main point is that we can still prove sufficiently strong bounds on the associated Green's function (as stated in \eqref{pam27.2} below) under the weaker and more natural assumptions \eqref{M01} on $M_0$. We provide all the details in section \ref{ProofGreen}.

We notice also that we could allow non-neutral perturbations in all the theorems, by modifying suitably the homogeneous equilibrium $M_0$ around which we perturb. 
\end{remark}

\begin{remark} We note that there is exponential loss of derivative during the nonlinear flow. For example, in Theorem \ref{MainThm} the initial data is assumed to be in the stronger Gevrey space $\mathcal{G}_{d'}^{\lambda_0,1/3}$, while the bounds on the solution are stated in the weaker Gevrey space $\mathcal{G}_{d'}^{\lambda_0/2,1/3}$. The precise amount of derivative loss can be reduced (for example from $\mathcal{G}_{d'}^{\lambda_0,1/3}$ to $\mathcal{G}_{d'}^{0.99\lambda_0,1/3}$, with small adjustments in the proof), but some loss is necessary to get time decay of macroscopic quantities. This is an important conceptual feature of Landau damping.
\end{remark}

\begin{remark} The global well-posedness theory of the final state problem and the construction of scattering operators (Theorems \ref{MainThm22} and \ref{MainThm33}) are new in the context of asymptotic stability of the Vlasov-Poisson system near homogeneous equilibria. To summarize, we prove that 3 nonlinear operators, the final state operator $f_0\mapsto g_\infty$, the wave operator $g_\infty\mapsto g(0)$, and the scattering operator $g_{-\infty}\mapsto g_\infty$ are well defined, Lipschitz continuous (between suitable spaces, with derivative loss, see \eqref{MainTHM4}, \eqref{MainTHMXY4}, \eqref{MainTH4}) and injective (in the quantitative sense given by the lower bounds in these inequalities).
\end{remark}

\subsection{Linear and nonlinear Landau damping}\label{LinNonlin} Landau damping is one of the most fundamental phenomena in Plasma Physics, and goes back to foundational work of Landau \cite{Lan1946} in 1946 on the linearized Vlasov-Poisson equations on $\R^3\times\R^3$ near Maxwellian equilibria. Nowadays, Landau damping is a term that is broadly used to denote a mechanism of decay of certain averaged quantities, as well as asymptotic stability. It took more than 60 years to develop a first understanding of Landau damping at the nonlinear level, for the full Vlasov-Poisson system \eqref{NVP}. This concerns the {\it{confined}} periodic setting $x\in\mathbb{T}^d$, where the decay of macroscopic quantities (such as the density $\rho$ and the electric field $E$) occurs due to phase mixing and relies mainly on the high regularity of the initial data (see also \cite{LZ2011}). After some preliminary works \cite{CM1998,HV2009}, nonlinear Landau damping was proved in pioneering work by Mouhot--Villani \cite{MouVil}, who showed that small and analytic perturbations of homogeneous equilibria $M_0$ that satisfy the Penrose criterion \eqref{eq:Penrose} lead to global solutions which scatter linearly, with density functions that decay exponentially in time. The theory has been refined significantly by Bedrossian-Masmoudi-Mouhot \cite{BMM2016} and then Grenier-Nguyen-Rodnianski \cite{GrNgRo}, who expanded the class of acceptable perturbations to the more flexible class of Gevrey spaces $\mathcal{G}_{d'}^{\lambda_0,s}$, in the almost optimal range $s>1/3$, and simplified the argument substantially by using energy estimates with suitably defined weights instead of a global Newton scheme.

Our Theorem \ref{MainThm} above represents a sharp endpoint result for this classical line of research on Landau damping in the confined case. Stability in Gevrey-3 spaces was already conjectured in \cite[Section 13]{MouVil}, and it is known by now that these spaces are best possible, due to the recent results in \cite{Be0,Be,Zil2021}. The theorem also applies to the vacuum case $M_0\equiv 0$; the normalization \eqref{VP00} is only needed to pass from \eqref{VPO} to \eqref{NVP}.

A second contribution of this paper is the development of a new and complete scattering theory for the Vlasov-Poisson system in the confined case around homogeneous equilibria, in Theorems \ref{MainThm22} and \ref{MainThm33}, also in optimal weighted Gevrey-3 spaces. The existence of the scattering operator (or even just the injectivity of the final state operator $f_0\mapsto g_\infty$ proved in Theorem \ref{MainThm}) has a natural physical interpretation: despite its name, ``Landau damping'' is different from other types of damping (for example damping of solutions of parabolic equations) in the sense that there is no loss of information during the nonlinear evolution in passing from the initial data $f_0$ to the final state $g_\infty$. The ``damping'' takes place only at the level of macroscopic quantities, like the density, consistent with the time-reversibility of the system.

In view of these results Landau damping is well understood by now in the confined case. We remark, however, that nonlinear Landau damping around homogeneous equilibria remains a major open problem in the \emph{unconfined case}, particularly in 3 dimensions $x\in\R^3$, both for the Vlasov-Poisson system and for the related relativistic Vlasov-Maxwell system. A key difference is that the Penrose condition \eqref{eq:Penrose} cannot hold uniformly as the frequency $\tau$ goes to $0$; as a result the Green's function cannot have fast decay and macroscopic quantities like the density and the electric field can only decay slowly, at dimension-dependent polynomial rates in $\langle t\rangle^{-1}$, which are not sufficient in dimension $d=3$. Linear stability results are known in these cases (see, for example, \cite{HKNgRo,BeMaMoLin,IoPaWaWiLinear,HKNgRo2} for recent work in this direction), indicating the possibility of nonlinear stability, but the only nonlinear result so far was proved recently by the authors in \cite{IoPaWaWiPoisson} for the Poisson equilibrium.

We conclude this subsection with a short and (far from exhaustive) discussion on related topics, and some additional references:

$\bullet\,\,$ Global existence of solutions of the Vlasov-Poisson system is classical in low dimensions, see \cite{BD1985,Pfa1992,LP1991}. We refer to \cite{Be} for a survey of recent results. The global dynamics of solutions, including the presence of logarithmically modified characteristics and modified scattering, is also well understood in the case of small initial data in $\R^3$ (which corresponds to the nonlinear stability of vacuum) using various methods, see \cite{CK2016, Smu2016, Wang2018, IoPaWaWiVacuum}.

$\bullet\,\,$ The construction of nonlinear wave operators and the scattering theory is a natural problem that has been extensively investigated for many dispersive and hyperbolic equations. In the context of collisionless kinetic equations, it has only been developed recently starting with the construction of the wave operator and the scattering map for the Vlasov-Poisson system near vacuum on $\R^3\times\R^3$ in \cite{FOPW21}. See also the more recent results in \cite{Big2022, Big2023, BiVe, PBA2024, ST2024}.

$\bullet\,\,$ In some variations of the Vlasov-Poisson system associated to the ion dynamics, one considers \textit{screened} interactions. In this case, the dynamics of low-frequencies is suppressed, which leads to favorable estimates and asymptotic stability \cite{GI2023}, also in the unconfined setting in 3D (see \cite{BeMaMoScreened,HKNgRoScreened}) and even in 2D (see \cite{HuNgXu}),  See also \cite{BeS, ChLuNg, FaRo, FaHoRo, HRSS2023, LiSt1, LZ2011, PW2020, PWY2022} for other stability results in related kinetic models. 

$\bullet\,\,$ Other classical plasma models include the hydrodynamic models of Euler-Poisson and Euler-Maxwell type, where some stability results are known in the unconfined Euclidean case (see \cite{Guo1998,GIP2016}), but no global stability results are known in the confined case. 

\subsection{Main ideas} Our proofs are based on what we call {\it{the $Z$-norm method}}. This is a general method that has been used extensively by the authors to prove global stability of various quasi-linear dispersive and hyperbolic equations, such as water waves, plasma models, and the Einstein equations of General Relativity. At the abstract level, it consists of proving simultaneous bootstrap control of two types of quantities:

(1) Energy functionals at the highest order of derivatives, sometimes including vector-fields that commute with the linearized equation;

(2) A special $Z$-norm (usually not based on $L^2$ spaces) which provides stronger control of the solutions, including time decay, but at a lower level of derivatives.

In our specific case, to prove Theorem \ref{MainThm} we use energy functionals defined in the Fourier space, of the form 
\begin{equation}\label{ide1}
\mathcal{E}^0_g(t):=\sum_{k\in \mathbb{Z}^d}\int_{\R^d}A(t,k,\xi)^2\Big\{\sum_{|a|\leq d'}\big|D^a_\xi\widehat{g}(t,k,\xi)\big|^2\Big\}\,d\xi,
\end{equation}
where $g$ is the profile of the solution and $A:[0,\infty)\times \Z^d\times\R^d\to [1,\infty)$ is a family of suitable weights. The precise choice of the weights $A(t,k,\xi)=e^{\lambda(t,|k,\xi|)}$ is very important, see the discussion in subsection \ref{NewW} below. These weights are decreasing in time, to reflect the "sliding regularity" basic mechanism of the proof, and increasing in frequency at a rate consistent with Gevrey-3 spaces, $\lambda(t,r)\approx \langle r\rangle^{1/3}$. 

We then define a matching $Z$-norm, which controls the macroscopic density, by the formula
\begin{equation}\label{ide2}
\mathcal{Z}^{0}_g(t):=\sup_{k\in \mathbb{Z}^d\setminus \{0\}}A(t,k,tk)\big|\widehat{g}(t,k,tk)\big||k|^{-1/2}=\sup_{k\in \mathbb{Z}^d\setminus \{0\}}A(t,k,tk)\big|\widehat{\rho}(t,k)\big||k|^{-1/2}.
\end{equation}
This type of $Z$-norms (weighted $L^\infty$ norms in the Fourier space) have been used extensively by the authors, starting with the work \cite{IoPu} on gravity water waves in 2D. A similar norm was also used in \cite{GrNgRo}, with a different weight $A$.

The basic idea is to prove uniform bootstrap control of the form 
\begin{equation}\label{ide3}
\mathcal{E}^0_{g}(t)+\big[\mathcal{Z}_{g}^{0}(t)\langle t\rangle^{6d}\big]^2\lesssim \kappa_0^2,
\end{equation}
for solutions $g$ of the system \eqref{NVP3}, where $\kappa_0$ denotes the size of the initial data. Notice the basic interplay of the proof: the energy norm is strong, but does not decay in time, while the $Z$-norm is comparatively weaker in terms of regularity (due to the factor of $|k|^{-1/2}$ in the definition), but is able to capture the time decay of the macroscopic density. The precise powers $1/2$ and $6d$ are related in the proof, but there is some flexibility in their choice.

The proof of the bootstrap estimates \eqref{ide3} consists of two main steps. First we use the density equation 
\begin{equation}\label{ide3.1}
\widehat{\rho}(t,k)+\int_{0}^t\widehat{\rho}(s,k)\widehat{M_0}((t-s)k)(t-s)\,ds=\widehat{\mathcal{N}}(t,k),
\end{equation}
where $\mathcal{N}$ is a suitable nonlinearity, together with bounds on Green's function associated with the equilibrium $M_0$, to control the $Z$-norm of the solution. Then we use the basic equation \eqref{NVP3} for $g$, in the Fourier space form \eqref{NVP4}, to prove uniform control on the functional $\mathcal{E}^0_{g}$. 

\subsection{New weights}\label{NewW} The main novelty in our construction, which allows us to reach the sharp Gevrey-3 regularity, is the choice of the weights $A$. The standard choice used in \cite{MouVil,BMM2016,GrNgRo} involves weights having favorable product structure of the form 
\begin{equation}\label{ide4}
A(t,k,\xi)=e^{\lambda(t,|k,\xi|)},\qquad \lambda(t,r)=\lambda_1(t)\lambda_2(r),
\end{equation}
for suitable functions $\lambda_1$ (decreasing) and $\lambda_2$ (increasing). This product structure easily leads to bilinear estimates of the form
\begin{equation}\label{ide5}
A(t,k_1+k_2,\xi_1+\xi_2)\lesssim A(t,k_1,\xi_1)A(t,k_2,\xi_2)\{(1+|k,\xi|)^{-4}+(1+|k',\xi'|)^{-4}\},
\end{equation}
for any $t,k,k',\xi,\xi'$, which play a crucial role in nonlinear analysis. Weights of this type have been used in \cite{BMM2016} and \cite{GrNgRo} (in the context of so-called "generator functions") to prove Landau damping in Gevrey spaces $\mathcal{G}_{d'}^{\lambda_0,s}$ in the full subcritical range $s>1/3$, but appear to be insufficient to deal with the critical case $s=1/3$.

Our main new idea is to work with a more general class of weights, of the form
\begin{equation}\label{ide6}
A(t,k,\xi)=e^{\lambda(t,|k,\xi|)},\qquad \lambda(t,r)=[\lambda_1+\delta(1+t)^{-\delta}]\langle r\rangle^{1/3}+\delta\big(1+t\langle r\rangle^{-2/3}\big)^{-\delta}\langle r\rangle^{1/3},
\end{equation}
where $\lambda_1\in[0.6\lambda_0,0.9\lambda_0]$ and $\delta=\lambda_0/200$. These weights do not have the basic product structure \eqref{ide4}, due to the last term, but we can still prove favorable bilinear estimates like \eqref{ide5} (in the stronger form given in \eqref{gu6}), as well as suitable commutator estimates like \eqref{gu6.5}, which are needed to prevent derivative loss. The key point is that we can choose the last term in the definition \eqref{ide6} in a way that improves the analysis of the density equation; for example we are able to prove uniform estimates like 
\begin{equation}\label{ide7}
\sum_{a\in\Z^d,\,|a|\leq |k|/8}\int_{0.99t}^{t}\frac{|t-s|A(t,k,tk)}{A(s,k,tk)}e^{-\delta(|a|^{1/3}+|sa+(t-s)k|^{1/3})}\,ds\lesssim 1,
\end{equation}
for any $t\in[10,\infty)$ and $k\in\Z^d\setminus\{0\}$ (see \eqref{pam16}), which ultimately allow us to close the main bootstrap argument at the level of the sharp Gevrey-3 spaces.

\subsection{The final state problem} At the quantitative level, the estimates we need to prove Theorem \ref{MainThm22} are similar to the estimates in the initial data problem, using a combination of an energy functional and a $Z$-norm. However we need a new family of weights  $A^\sharp(t,k,\xi)=e^{\lambda^\sharp(t,|k,\xi|)}:[0,\infty)\times \Z^d\times\R^d\to [1,\infty)$, which are now increasing in time to indicate that the solution comes from infinity and loses regularity as $t$ decreases. We also need a new density equation, of the form
\begin{equation}\label{ide8}
\widehat{\rho}(t,k)-\int_t^{\infty}\widehat{\rho}(s,k)\widehat{M_0}((t-s)k)(t-s)\,ds=\widehat{\mathcal{N}'}(t,k),
\end{equation}
for a suitable nonlinearity $\mathcal{N}'$, which replaces the equation \eqref{ide3.1} and gives us a basic formula for the density $\rho$ at time $t$ in terms of $\rho$ at later times.

To construct the solution itself we construct a family of solutions $g_n$ of the system \eqref{NVP3} on the time interval $[0,n]$, with data $g_n(n)=g_\infty$. Then we prove estimates on the differences of these solutions, and find the full solution $g$ as the limit as $n\to \infty$ of $g_n$. This construction is possible as long as we allow partial loss of regularity, as claimed in Theorem \ref{MainThm22}, and provided that we control very well differences of solutions.

\subsection{Organization} The rest of the paper is organized as follows: in section \ref{BasCon} we derive our main equations for the profile $g$ and the density $\rho$, define our two families of weights and prove some of their basic properties, and state our two main bootstrap propositions, Propositions \ref{MainBootstrap} and \ref{MainBootstrapInfin}. Then we prove these propositions in sections \ref{pam1} and \ref{pamA1} respectively. In section 5 we show how to use these bootstrap propositions to complete the proofs of the main Theorems \ref{MainThm}, \ref{MainThm22}, and \ref{MainThm33}. Finally, in section 6 we prove Lemma \ref{GreenF}, which provides suitable bounds on the Green's function associated to the equilibrium $M_0$.

\section{The main equations and the bootstrap proposition}\label{BasCon}

Recall the main equations \eqref{NVP3} for the kinetic profile $g$. We assume that $T_1\leq T_2\in[0,\infty)$ and $g\in C([T_1,T_2]:\mathcal{G}^{\lambda_0/2,1/3}_{d'}(\T^d\times\R^d))$ is a solution of the system \eqref{NVP3}, satisfying the neutrality condition $\int_{\T^d\times\R^d}g(t,x,v)\,dxdv=0$. Taking the Fourier transform in both $x$ and $v$ variables we derive the equations
\begin{equation}\label{NVP4}
\begin{split}
&\partial_t\widehat{g}(t,k,\xi)+\widehat{\rho}(t,k)\widehat{M_0}(\xi-tk)\frac{k\cdot (\xi-tk)}{|k|^2}+\frac{1}{(2\pi)^d}\sum_{l\in\Z^d}\widehat{\rho}(t,l)\widehat{g}(t,k-l,\xi-tl)\frac{l\cdot(\xi-tk)}{|l|^2}=0,\\
&\widehat{\rho}(t,k):=\widehat{g}(t,k,tk).
\end{split}
\end{equation}

\subsection{The initial data problem} It follows from \eqref{NVP4} that
\begin{equation*}
\begin{split}
\widehat{g}(t,k,\xi)-\widehat{g}(T_1,k,\xi)&+\int_{T_1}^t\widehat{\rho}(s,k)\widehat{M_0}(\xi-sk)\frac{k\cdot (\xi-sk)}{|k|^2}\,ds\\
&+\frac{1}{(2\pi)^d}\sum_{l\in\Z^d}\int_{T_1}^t\widehat{\rho}(s,l)\widehat{g}(s,k-l,\xi-sl)\frac{l\cdot(\xi-sk)}{|l|^2}\,ds=0.
\end{split}
\end{equation*}
Letting $\xi=tk$ it follows that
\begin{equation}\label{NVP6}
\begin{split}
&\widehat{\rho}(t,k)+\int_{T_1}^t\widehat{\rho}(s,k)\widehat{M_0}((t-s)k)(t-s)\,ds=\widehat{\mathcal{N}}(t,k),\\
&\widehat{\mathcal{N}}(t,k):=\widehat{g}(T_1,k,tk)-\frac{1}{(2\pi)^d}\sum_{l\in\Z^d\setminus\{0\}}\int_{T_1}^t\widehat{\rho}(s,l)\widehat{g}(s,k-l,tk-sl)\frac{(t-s)l\cdot k}{|l|^2}\,ds.
\end{split}
\end{equation}

We would like to use the Volterra equation \eqref{NVP6} and the main assumptions \eqref{M01}--\eqref{eq:Penrose} to prove an important formula for $\rho$. We extend $\widehat{\rho}$ to a function on $\R\times\Z^d$ supported in $[T_1,T_2]\times\Z^d$, and extend the function $\widehat{\mathcal{N}}$ according to the identity in the first line of \eqref{NVP6} (clearly $\widehat{\mathcal{N}}(t,k)=0$ if $t<T_1$). Then we define, for $\tau\in\R$ and $k\in\Z^d$,
\begin{equation}\label{pam22.1}
\widetilde{\rho}(\tau,k):=\int_\R\widehat{\rho}(t,k)e^{-it\tau}\,dt,\qquad \widetilde{\mathcal{N}}(\tau,k):=\int_\R\widehat{\mathcal{N}}(t,k)e^{-it\tau}\,dt.
\end{equation}

For $\tau\in\mathbb{H}_-:=\{\tau\in\C:\,\Im\tau\leq 0\}$ and $k\in\Z^d\setminus\{0\}$ let
\begin{equation}\label{gu14}
L(\tau,k):=\int_0^\infty e^{-i\tau s}s\widehat{M_0}(sk)\,ds.
\end{equation}
In view of \eqref{M01}, for any $k\in\Z^d\setminus\{0\}$ the function $L(.,k)$ is well-defined and continuous in $\mathbb{H}_-$, and analytic in the open half-plane $\mathbb{H}^\circ_-:=\{\tau\in\C:\,\Im\tau< 0\}$.

It follows from the identity in the first line of \eqref{NVP6} that
\begin{equation}\label{pam22}
\widetilde{\rho}(\tau,k)\big(1+L(\tau,k)\big)=\widetilde{\mathcal{N}}(\tau,k),\qquad\tau\in\R,\,k\in\Z^d\setminus\{0\},
\end{equation}
where the function $L$ is defined as in \eqref{gu14}. For $\tau\in\mathbb{H}_-$ and $k\in\Z^d\setminus\{0\}$ let 
\begin{equation}\label{gu18}
L'(\tau,k):=\frac{L(\tau,k)}{1+L(\tau,k)}.
\end{equation}
In view of the Penrose condition \eqref{eq:Penrose}, the function $L':\mathbb{H}_-\times (\Z^d\setminus\{0\})\to\C$ is well-defined and continuous in $\mathbb{H}_-$, and analytic in $\mathbb{H}^\circ_-$. Moreover, using \eqref{pam22},
\begin{equation}\label{pam26}
\widetilde{\rho}(\tau,k)=\widetilde{\mathcal{N}}(\tau,k)-\widetilde{\mathcal{N}}(\tau,k)L'(\tau,k),\qquad\tau\in\R,\,k\in\Z^d\setminus\{0\}.
\end{equation}
We take the inverse Fourier transform in $\tau$ to conclude that
\begin{equation}\label{gu20.5}
\begin{split}
&\widehat{\rho}(t,k)=\widehat{\mathcal{N}}(t,k)-\int_\R\widehat{\mathcal{N}}(t-s,k)\widehat{G}(s,k)\,ds\quad\text{ where }\quad\widehat{G}(s,k):=\frac{1}{2\pi}\int_\R e^{is\tau}L'(\tau,k)\,d\tau.
\end{split}
\end{equation}

Our next lemma provides estimates on the functions $L,L'$ and the Green's function $\widehat{G}$, which show in particular that the formal calculations above are justified.

\begin{lemma}\label{GreenF} (i) The functions $L,L':\mathbb{H}_-\times(\Z^d\setminus\{0\})\to\C$ defined in \eqref{gu14} and \eqref{gu18} are well-defined and continuous in $\mathbb{H}_-$, analytic in $\mathbb{H}^\circ_-$, and satisfy the bounds
\begin{equation}\label{gu19.1}
\begin{split}
\big\|\langle \alpha\rangle L(\alpha+i\gamma,k)\big\|_{L^2_\alpha}+\big\|\langle \alpha\rangle L'(\alpha+i\gamma,k)\big\|_{L^2_\alpha}&\leq C_0^\ast |k|^{-1/2},\\
\big\|\langle\alpha\rangle D^a_\alpha L(\alpha+i\gamma,k)\big\|_{L^\infty_\alpha}+\big\|\langle\alpha\rangle D^a_\alpha L'(\alpha+i\gamma,k)\big\|_{L^\infty_\alpha}&\leq C_0^\ast \frac{(3a)!(1.01)^{a}}{(|k|\lambda_0^3)^{a+1}},
\end{split}
\end{equation}
for any $k\in\Z^d\setminus\{0\}$, $\gamma\in(-\infty,0]$, and $a\in\Z_+$. The constant $C_0^\ast$ may depend only on the parameters $d$, $\lambda_0$, and $\vartheta$.

(ii) With $\widehat{\rho}$ and $\widehat{\mathcal{N}}$ as in \eqref{NVP6}, for any $t\in\R$ and $k\in\Z^d\setminus\{0\}$ we have
\begin{equation}\label{pam27.1}
\widehat{\rho}(t,k)=\widehat{\mathcal{N}}(t,k)-\int_{T_1}^t\widehat{\mathcal{N}}(s,k)\widehat{G}(t-s,k)\,ds
\end{equation}
where the Green's function $\widehat{G}:\R\times(\Z^d\setminus\{0\})\to\C$ defined as in \eqref{gu20.5} is supported in $[0,\infty)\times (\Z^d\setminus\{0\})$ and satisfies the uniform bounds
\begin{equation}\label{pam27.2}
\big|\widehat{G}(s,k)\big|\leq C_0^\ast e^{-0.95\lambda_0|sk|^{1/3}}.
\end{equation}
\end{lemma}

We provide a self-contained proof of this lemma in section \ref{ProofGreen}. 

\subsubsection{Energy functionals and the first bootstrap proposition}\label{weightsdef} The main idea is to estimate the increment of suitable functionals, which are defined using special weights. For $t,r\in\R_+$, $k,\xi\in\R^d$ we define
\begin{equation}\label{gu1}
\begin{split}
&\lambda(t,r):=\lambda_1\langle r\rangle^{1/3}+\delta(1+t)^{-\delta}\langle r\rangle^{1/3}+\delta\big(1+t\langle r\rangle^{-2/3}\big)^{-\delta}\langle r\rangle^{1/3},\\
&A(t,k,\xi):=e^{\lambda(t,|k,\xi|)},
\end{split}
\end{equation}
where $\lambda_1\in[0.5\lambda_0,0.9\lambda_0]$, $\delta:=\lambda_0/200$ and $\langle w\rangle:=(1+|w|^2)^{1/2}$, $w\in\R^n$.

The precise definition of the weight $A$ is important; some properties are proved in section \ref{TotalWeights} below. Notice that the weights $\lambda$ and $A$ are decreasing in $t$. 

Our first bootstrap argument is based on controlling simultaneously an energy functional for the profile $g$ and a decaying norm for the density $\rho$. Let $d'=\lfloor d/2+1\rfloor$ and define the spaces
\begin{equation}\label{XTspace}
\begin{split}
&X_{[T_1,T_2]}:=\Big\{h\in C([T_1,T_2]:\mathcal{G}^{\lambda_0/2,1/3}_{d'}(\T^d\times\R^d)):\\
&\quad\|h\|_{X_{[T_1,T_2]}}:=\sup_{t\in[T_1,T_2]}\Big\{\sum_{k\in \mathbb{Z}^d}\int_{\R^d}A(t,k,\xi)^2\Big [\sum_{|a|\leq d'}\big|D^a_\xi\widehat{h}(t,k,\xi)\big|^2\Big ]\,d\xi\Big\}^{1/2}<\infty\Big\},
\end{split}
\end{equation}
for $T_1\leq T_2\in [0,\infty)$. For any $h\in X_{[T_1,T_2]}$ and $t\in[T_1,T_2]$, $\beta\in[0,1]$, and $p\in[0,2]$ we define
\begin{equation}\label{rec1}
\begin{split}
\mathcal{E}^p_h(t)&:=\sum_{k\in \mathbb{Z}^d}\int_{\R^d}\langle k,\xi\rangle^{-2p}A(t,k,\xi)^2\Big\{\sum_{|a|\leq d'}\big|D^a_\xi\widehat{h}(t,k,\xi)\big|^2\Big\}\,d\xi,\\
\mathcal{Z}^{p}_h(t)&:=\sup_{k\in \mathbb{Z}^d\setminus \{0\}}\langle k,tk\rangle^{-p}A(t,k,tk)\big|\widehat{h}(t,k,tk)\big||k|^{-\beta}.
\end{split}
\end{equation}
There are two parameters in these definitions: the parameter $\beta$ is taken $\beta=1/2$ in the proof, and it indicates the fact that the $Z$-norms of solutions are measured at lower order of differentiability than the energy norms. The parameter $p$ is taken either $p=0$ to estimate the solutions themselves or $p=2$ to estimate the difference of solutions (due to some loss of symmetry, the difference of solutions can only be estimated in slightly weaker norms).

Our first main bootstrap proposition is the following: 

\begin{proposition}\label{MainBootstrap}
Assume $T_1\leq T_2\in[0,\infty)$ and $g_1,g_2\in X_{[T_1,T_2]}$ are real-valued solutions of the system \eqref{NVP3} satisfying $\widehat{g_i}(t,0,0)=0$ for any $t\in[T_1,T_2]$ and $i\in\{1,2\}$. Assume that $\eps_0\leq\overline{\eps}\ll 1$, $\lambda_1\in[0.5\lambda_0,0.9\lambda_0]$, and
\begin{equation}\label{boot1}
\|g_i(T_1)\|_{{\mathcal{G}_{d'}^{\lambda_1+4\delta,1/3}}}\leq\overline\eps,\qquad \|(g_1-g_2)(T_1)\|_{\mathcal{G}_{d'}^{\lambda_1+4\delta,1/3}}\leq\eps_0,
\end{equation}
for $i\in\{1,2\}$. Let $\beta=1/2$ and assume that for any $t\in[T_1,T_2]$ and $i\in\{1,2\}$
\begin{equation}\label{boot2}
\begin{split}
&\mathcal{E}^0_{g_i}(t)+\big[\mathcal{Z}_{g_i}^{0}(t)\langle t\rangle^{6d}\big]^2\leq 4\overline{C}^2\overline{\eps}^2,\\
&\mathcal{E}^2_{g_1-g_2}(t)+\big[\mathcal{Z}_{g_1-g_2}^{2}(t)\langle t\rangle^{6d}\big]^2\leq 4\overline{C}^2\eps_0^2,
\end{split}
\end{equation}
where $\overline{C}\leq\overline{\eps}^{-1/8}$ is a sufficiently large constant that depends only on the structural constants $d,\lambda_0,\vartheta$. Then for any $t\in[T_1,T_2]$ and $i\in\{1,2\}$ we have the improved bounds
\begin{equation}\label{boot3}
\begin{split}
&\mathcal{E}^0_{g_i}(t)+\big[\mathcal{Z}_{g_i}^{0}(t)\langle t\rangle^{6d}\big]^2\leq \overline{C}^2\overline{\eps}^2,\\
&\mathcal{E}_{g_1-g_2}^2(t)+\big[\mathcal{Z}_{g_1-g_2}^{2}(t)\langle t\rangle^{6d}\big]^2\leq \overline{C}^2\eps_0^2.
\end{split}
\end{equation}
\end{proposition}

This proposition plays a key role in the proof of Theorem \ref{MainThm} (i); we will prove it in section \ref{pam1} below. We note that the constant $\overline{C}$ will be taken much larger than other structural constants such as the constants $C_0^\ast$ in Lemma \ref{GreenF}, $C_0$ in Lemma \ref{g*}, $C_1$ in Lemma \ref{LemN}, and $C_2$ in Proposition \ref{MainBootstrapIm1}. 

\subsection{The final state problem}

We also prove quantitative estimates going backwards in time. Assume that $T_1\leq T_2\in[0,\infty)$ and $g\in C([T_1,T_2]:\mathcal{G}_{d'}^{\lambda_0/2,1/3}(\T^d\times\R^d))$ is a solution of the system \eqref{NVP3}. For any $t\in[T_1,T_2]$ we use \eqref{NVP4} and integrate for $s\in[t,T_2]$, thus
\begin{equation*}
\begin{split}
-\widehat{g}(t,k,\xi)+\widehat{g}(T_2,k,\xi)&+\int_t^{T_2}\widehat{\rho}(s,k)\widehat{M_0}(\xi-sk)\frac{k\cdot (\xi-sk)}{|k|^2}\,ds\\
&+\frac{1}{(2\pi)^d}\sum_{l\in\Z^d}\int_t^{T_2}\widehat{\rho}(s,l)\widehat{g}(s,k-l,\xi-sl)\frac{l\cdot(\xi-sk)}{|l|^2}\,ds=0.
\end{split}
\end{equation*}
Letting $\xi=tk$ it follows that
\begin{equation}\label{infin0}
\begin{split}
&\widehat{\rho}(t,k)-\int_t^{T_2}\widehat{\rho}(s,k)\widehat{M_0}((t-s)k)(t-s)\,ds=\widehat{\mathcal{N}'}(t,k),\\
&\widehat{\mathcal{N}'}(t,k):=\widehat{g}(T_2,k,tk)+\frac{1}{(2\pi)^d}\sum_{l\in\Z^d\setminus\{0\}}\int_t^{T_2}\widehat{\rho}(s,l)\widehat{g}(s,k-l,tk-sl)\frac{(t-s)l\cdot k}{|l|^2}\,ds.
\end{split}
\end{equation}

We would like to find an integral formula for $\rho$ in terms of $\mathcal{N}'$. For this we define $\rho_1(t)=\rho(t)$ for $t\in[T_1,T_2]$ and $\rho_1(t)=0$ if $t\notin[T_1,T_2]$, and then we define $\mathcal{N}'_1$ by
\begin{equation}\label{infin1}
\widehat{\mathcal{N}'_1}(t,k):=\widehat{\rho_1}(t,k)-\int_t^\infty\widehat{\rho_1}(s,k)\widehat{M_0}((t-s)k)(t-s)\,ds.
\end{equation}
For $\tau\in\R$ and $k\in\Z^d\setminus\{0\}$ we define the time Fourier transforms
\begin{equation*}
\widetilde{\rho_1}(\tau,k):=\int_\R\widehat{\rho_1}(t,k)e^{-it\tau}\,dt,\qquad \widetilde{\mathcal{N}'_1}(\tau,k):=\int_\R\widehat{\mathcal{N}'_1}(t,k)e^{-it\tau}\,dt.
\end{equation*}
It follows from \eqref{infin1} that for any $\tau\in\R$ and $k\in\Z^d\setminus\{0\}$ we have
\begin{equation*}
\widetilde{\mathcal{N}'_1}(\tau,k)=\widetilde{\rho_1}(\tau,k)(1+L(-\tau,-k)),
\end{equation*}
where $L$ is defined as in \eqref{gu14}. Therefore
\begin{equation*}
\widetilde{\rho_1}(\tau,k)=\widetilde{\mathcal{N}'_1}(\tau,k)-\widetilde{\mathcal{N}'_1}(\tau,k)L'(-\tau,-k),\qquad\tau\in\R,\,k\in\Z^d\setminus\{0\},
\end{equation*}
where $L'$ is defined as in \eqref{gu18}. We take the inverse Fourier transform and recall the definition \eqref{gu20.5} of the Green's functions $\widehat{G}$, so
\begin{equation*}
\widehat{\rho_1}(t,k)=\widehat{\mathcal{N}'_1}(t,k)-\int_\R\widehat{\mathcal{N}'_1}(s,k)\widehat{G}(s-t,-k)\,ds.
\end{equation*}
Recall that $\widehat{G}$ is supported in $[0,\infty)\times(\Z^d\setminus\{0\})$ (see Lemma \ref{GreenF}), $\rho_1(t)=\rho(t)$, $\mathcal{N}'_1(t)=\mathcal{N}'(t)$ for $t\in[T_1,T_2]$, and $\mathcal{N}'_1(t)=0$ for $t>T_2$. We can thus derive our main identity
\begin{equation}\label{infin2}
\widehat{\rho}(t,k)=\widehat{\mathcal{N}'}(t,k)-\int_t^{T_2}\widehat{\mathcal{N}'}(s,k)\widehat{G}(s-t,-k)\,ds\qquad\text{ for }t\in[T_1,T_2]\text{ and }k\in\Z^d\setminus\{0\}.
\end{equation}
This formula serves as the main substitute of the identity \eqref{pam27.1} in the analysis of the final state problem.

\subsubsection{The second bootstrap proposition} 

To understand the evolution backwards in time we need to work with a new set of weights, which are increasing in time. We define these weights as in \eqref{gu1}: for $t\in [0,\infty)$, $r\in\R_+$, $k,\xi\in\R^d$ we define
\begin{equation}\label{infin3}
\begin{split}
&\lambda^\sharp(t,r):=\lambda_1\langle r\rangle^{1/3}-\delta(1+t)^{-\delta}\langle r\rangle^{1/3}-\delta\big(1+t\langle r\rangle^{-2/3}\big)^{-\delta}\langle r\rangle^{1/3},\\
&A^\sharp(t,k,\xi):=e^{\lambda^\sharp(t,|k,\xi|)},
\end{split}
\end{equation}
where $\lambda_1\in[0.5\lambda_0,0.9\lambda_0]$ and $\delta=\lambda_0/200$. We will use these weights for certain values of $\lambda_1$.

Our second bootstrap argument is also based on controlling simultaneously an energy functional for $g$ and a decaying norm for $\rho$. As in subsection \ref{weightsdef} we define
\begin{equation}\label{infin4}
\begin{split}
&X^\sharp_{[T_1,T_2]}:=\Big\{h\in C([T_1,T_2]:\mathcal{G}^{\lambda_0/2,1/3}_{d'}(\T^d\times\R^d)):\\
&\qquad\|h\|_{X^\sharp_{[T_1,T_2]}}:=\sup_{t\in[T_1,T_2]]}\Big\{\sum_{k\in \mathbb{Z}^d}\int_{\R^d}A^\sharp(t,k,\xi)^2\Big [\sum_{|a|\leq d'}\big|D^a_\xi\widehat{h}(t,k,\xi)\big|^2\Big ]\,d\xi\Big\}^{1/2}<\infty\Big\}.
\end{split}
\end{equation}
For any $h\in X^\sharp_{[T_1,T_2]}$, $t\in[T_1,T_2]$, and $p\in[0,2]$ we define
\begin{equation}\label{infin5}
\begin{split}
\mathcal{E}^{\sharp,p}_{h}(t)&:=\sum_{k\in \mathbb{Z}^d}\int_{\R^d}\langle k,\xi\rangle^{-2p}A^\sharp(t,k,\xi)^2\Big\{\sum_{|a|\leq d'}\big|D^a_\xi\widehat{h}(t,k,\xi)\big|^2\Big\}\,d\xi,\\
\mathcal{Z}^{\sharp,p}_{h}(t)&:=\sup_{k\in \mathbb{Z}^d\setminus \{0\}}\langle k,tk\rangle^{-p}A^\sharp(t,k,tk)\big|\widehat{h}(t,k,tk)\big||k|^{-1/2}.
\end{split}
\end{equation}

Our second main bootstrap proposition is the following: 

\begin{proposition}\label{MainBootstrapInfin}
Assume $T_1\leq T_2\in[0,\infty)$ and $g_1,g_2\in X^\sharp_{[T_1,T_2]}$ are real-valued solutions of the system \eqref{NVP3} satisfying $\widehat{g_i}(t,0,0)=0$ for any $t\in[T_1,T_2]$ and $i\in\{1,2\}$. Assume that $\theta_0\leq\overline{\theta}\ll 1$ and
\begin{equation}\label{infin6}
\|g_i(T_2)\|_{{\mathcal{G}_{d'}^{\lambda_1+4\delta,1/3}}}\leq\overline\theta,\qquad \|(g_1-g_2)(T_2)\|_{\mathcal{G}_{d'}^{\lambda_1+4\delta,1/3}}\leq\theta_0.
\end{equation}
for $i\in\{1,2\}$, where $\lambda_1\in[0.5\lambda_0,0.9\lambda_0]$. Assume that 
\begin{equation}\label{infin7}
\begin{split}
&\mathcal{E}^{\sharp,0}_{g_i}(t)+\big[\mathcal{Z}_{g_i}^{\sharp,0}(t)\langle t\rangle^{6d}\big]^2\leq 4\overline{B}^2\overline{\theta}^2,\\
&\mathcal{E}^{\sharp,2}_{g_1-g_2}(t)+\big[\mathcal{Z}_{g_1-g_2}^{\sharp,2}(t)\langle t\rangle^{6d}\big]^2\leq 4\overline{B}^2\theta_0^2,
\end{split}
\end{equation}
for any $t\in[T_1,T_2]$ and $i\in\{1,2\}$, where $\overline{B}\leq\overline{\theta}^{-1/8}$ is a sufficiently large constant. Then for any $t\in[T_1,T_2]$ and $i\in\{1,2\}$ we have the improved bounds
\begin{equation}\label{infin8}
\begin{split}
&\mathcal{E}^{\sharp,0}_{g_i}(t)+\big[\mathcal{Z}_{g_i}^{\sharp,0}(t)\langle t\rangle^{6d}\big]^2\leq \overline{B}^2\overline{\theta}^2,\\
&\mathcal{E}^{\sharp,2}_{g_1-g_2}(t)+\big[\mathcal{Z}_{g_1-g_2}^{\sharp,2}(t)\langle t\rangle^{6d}\big]^2\leq \overline{B}^2\theta_0^2.
\end{split}
\end{equation}
\end{proposition}

\subsection{Some lemmas}\label{TotalWeights} 

We prove first some properties of the main weights $A$ and $A^\sharp$.

\begin{lemma}\label{lambdaLem}
(i) Assume that $\lambda_1\in[\lambda_0/2,\lambda_0]$. For any $t,r\in[0,\infty)$ we have
\begin{equation}\label{gu3.1}
\lambda(t,r),\lambda^\sharp(t,r)\in[(\lambda_1-2\delta)\langle r\rangle^{1/3},(\lambda_1+2\delta)\langle r\rangle^{1/3}],
\end{equation}
\begin{equation}\label{gu3.3}
-\partial_t\lambda(t,r)=\partial_t\lambda^\sharp(t,r)=\delta^2(1+t)^{-\delta-1}\langle r\rangle^{1/3}+\delta^2\big(1+t\langle r\rangle^{-2/3}\big)^{-\delta-1}\langle r\rangle^{-1/3},
\end{equation}
\begin{equation}\label{gu3.2}
\begin{split}
\partial_r\lambda(t,r)&=\big[\lambda_1+\delta(1+t)^{-\delta}\big](r/3)\langle r\rangle^{-5/3}\\
&+\delta\big(1+t\langle r\rangle^{-2/3}\big)^{-\delta-1}(r/3)\langle r\rangle^{-5/3}\big[1+t\langle r\rangle^{-2/3}(1+2\delta)\big],\\
\partial_r\lambda^\sharp(t,r)&=\big[\lambda_1-\delta(1+t)^{-\delta}\big](r/3)\langle r\rangle^{-5/3}\\
&-\delta\big(1+t\langle r\rangle^{-2/3}\big)^{-\delta-1}(r/3)\langle r\rangle^{-5/3}\big[1+t\langle r\rangle^{-2/3}(1+2\delta)\big],\\
\end{split}
\end{equation}
In particular the functions $\lambda$ and $\lambda^\sharp$ are increasing in $r$ and
\begin{equation}\label{gu2}
\partial_r\lambda(t,r),\partial_r\lambda^\sharp(t,r)\in[(\lambda_1/3-\delta)r\langle r\rangle^{-5/3},(\lambda_1/3+\delta)r\langle r\rangle^{-5/3}]\,\,\text{ for any }\,\,t,r\in [0,\infty). 
\end{equation}
Moreover, if $a\leq b\in\R_+$ and $x,y\in\R_+$ then
\begin{equation}\label{gu4}
\begin{split}
&\lambda(a,r)-\lambda(b,r)=\lambda^\sharp(b,r)-\lambda^\sharp(a,r)\\
&=\delta^2\int_a^b\big[(1+u)^{-\delta-1}\langle r\rangle^{1/3}+\big(1+u\langle r\rangle^{-2/3}\big)^{-\delta-1}\langle r\rangle^{-1/3}\big]\,du.
\end{split}
\end{equation}

(ii) If $k,\xi,k',\xi'\in\R^d$ and $t\in[0,\infty)$ then
\begin{equation}\label{gu6}
\begin{split}
A(t,k+k',\xi+\xi')&\leq A(t,k,\xi)A(t,k',\xi')e^{-(\lambda_1/4)\min(\langle k,\xi\rangle,\langle k',\xi'\rangle)^{1/3}},\\
A^\sharp(t,k+k',\xi+\xi')&\leq A^\sharp(t,k,\xi)A^\sharp(t,k',\xi')e^{-(\lambda_1/4)\min(\langle k,\xi\rangle,\langle k',\xi'\rangle)^{1/3}}.
\end{split}
\end{equation}
Moreover, if $k,\xi,k',\xi'\in\R^d$, $t\in[0,\infty)$, and $|k',\xi'|\leq|k,\xi|/8$ then
\begin{equation}\label{gu6.5}
\begin{split}
\frac{|A(t,k+k',\xi+\xi')-A(t,k,\xi)|}{A(t,k,\xi)A(t,k',\xi')}&\lesssim e^{-(\lambda_1/4)\langle k',\xi'\rangle^{1/3}}\langle k,\xi\rangle^{-2/3},\\
\frac{|A^\sharp(t,k+k',\xi+\xi')-A^\sharp(t,k,\xi)|}{A^\sharp(t,k,\xi)A^\sharp(t,k',\xi')}&\lesssim e^{-(\lambda_1/4)\langle k',\xi'\rangle^{1/3}}\langle k,\xi\rangle^{-2/3}.
\end{split}
\end{equation}
\end{lemma}

\begin{proof} (i) The identities and the bounds follow directly from the definitions.

(ii) In view of \eqref{gu2}, to prove the bounds in the first line of \eqref{gu6} it suffices to show that if $t,y,b\in[0,\infty)$ and $b\leq y$ then
\begin{equation}\label{gu7}
\lambda(t,y+b)\leq \lambda(t,y)+\lambda(t,b)-(\lambda_1/4)\langle b\rangle^{1/3}.
\end{equation}
To prove this we notice that $\partial_r\lambda(t,r)\leq (\lambda_1/2)\langle r\rangle^{-2/3}$ for any $r\geq 0$, due to \eqref{gu2}, thus
\begin{equation*}
\lambda(t,y+b)-\lambda(t,y)\leq (\lambda_1/2)\langle y\rangle^{-2/3}b\leq (\lambda_1/2)\langle b\rangle^{1/3}\leq\lambda(t,b)-(\lambda_1/4)\langle b\rangle^{1/3}
\end{equation*}
as desired. The bounds in the second line of \eqref{gu6} follow in a similar way.

To prove the bounds in the first line of \eqref{gu6.5} it suffices to prove that 
\begin{equation}\label{gu7.5}
e^{\lambda(t,y+b)-\lambda(t,y)}-1\lesssim e^{(3\lambda_1/5)\langle b\rangle^{1/3}}\langle y\rangle^{-2/3},
\end{equation}
provided that $t,y,b\in[0,\infty)$ and $b\leq y$. This follows easily since $\lambda(t,y+b)-\lambda(t,y)\leq (2\lambda_1/5)\langle y\rangle^{-2/3}b$. The bounds in the second line of \eqref{gu6.5} follow in a similar way.
\end{proof}

We prove now two pointwise estimates in the Fourier space.

\begin{lemma}\label{g*}
Assume $T_1\leq T_2\in[0,\infty)$, $h_1\in X_{[T_1,T_2]}$, $h_2\in X^\sharp_{[T_1,T_2]}$. For any $p\in[0,2]$, $t\in[T_1,T_2]$, and $k\in\Z^d$ let
\begin{equation}\label{gu10}
\begin{split}
H_p(h_1)(t,k)&:=\sup_{\xi\in\R^d}\big|\langle k,\xi\rangle^{-p}A(t,k,\xi)\widehat{h_1}(t,k,\xi)\big|,\\
H^{\sharp}_p(h_2)(t,k)&:=\sup_{\xi\in\R^d}\big|\langle k,\xi\rangle^{-p}A^\sharp(t,k,\xi)\widehat{h_2}(t,k,\xi)\big|.
\end{split}
\end{equation}
Then there is a constant $C_0=C_0(d)$ such that 
\begin{equation}\label{gu11}
\begin{split}
|H_p(h_1)(t,k)|^2&\leq C_0^2\int_{\R^d}\langle k,\xi\rangle^{-2p}A(t,k,\xi)^2\Big\{\sum_{|a|\leq d'}\big|D^a_\xi\widehat{h_1}(t,k,\xi)\big|^2\Big\}\,d\xi,\\
|H_p^\sharp(h_2)(t,k)|^2&\leq C_0^2\int_{\R^d}\langle k,\xi\rangle^{-2p}A^\sharp(t,k,\xi)^2\Big\{\sum_{|a|\leq d'}\big|D^a_\xi\widehat{h_2}(t,k,\xi)\big|^2\Big\}\,d\xi.
\end{split}
\end{equation}
\end{lemma}

\begin{proof} 
By Sobolev embedding, 
\begin{equation*}
|H_p(h_1)(t,k)|^2\lesssim_d \int_{\R^d}\sum_{|a|\leq d'}\big|D^a_\xi[\langle k,\xi\rangle^{-p}A(t,k,\xi)\widehat{h_1}(t,k,\xi)]\big|^2\,d\xi.
\end{equation*}
Using the definition \eqref{gu1} it is easy to see that
\begin{equation*}
|D^a_\xi A(t,k,\xi)|\lesssim A(t,k,\xi)\qquad\text{ for any }t\in[T_1,T_2],\,k,\xi\in\R^d,\,|a|\leq d'.
\end{equation*}
The bounds in the first line of \eqref{gu11} follow. The bounds in the second line follow similarly.
\end{proof}

\section{Proof of Proposition \ref{MainBootstrap}}\label{pam1}

In this section we prove our first main bootstrap proposition.

\subsection{Improved control on the $Z$-norm} We prove first the improved estimates \eqref{boot3} on the quantities $\mathcal{Z}_{g_i}^{0}$ and $\mathcal{Z}_{g_1-g_2}^{2}$, using the identities \eqref{pam27.1}. We start with estimates on the nonlinearities $\mathcal{N}_1$ and $\mathcal{N}_2$.

\begin{lemma}\label{LemN}
Assume that $g_1,g_2\in X_{[T_1,T_2]}$ are solutions satisfying the hypothesis in Proposition \ref{MainBootstrap} (recall $\beta=1/2$), let $\widehat{\rho_i}(t,k):=\widehat{g_i}(t,k,tk)$, and define $\mathcal{N}_i$ as in \eqref{NVP6}. Then
\begin{equation}\label{pam2}
\begin{split}
\sup_{t\in[T_1,T_2],\,k\in\Z^d\setminus\{0\}}A(t,k,tk)|\widehat{\mathcal{N}_i}(t,k)||k|^{-\beta}\langle t\rangle^{6d}&\leq C_1\overline{\eps},\\
\sup_{t\in[T_1,T_2],\,k\in\Z^d\setminus\{0\}}\langle k,tk\rangle^{-2}A(t,k,tk)|(\widehat{\mathcal{N}_1}-\widehat{\mathcal{N}_2})(t,k)||k|^{-\beta}\langle t\rangle^{6d}&\leq C_1\eps_0,
\end{split}
\end{equation}
for $i\in\{1,2\}$ and a constant $C_1=C_1(d,\lambda_0)$ sufficiently large.
\end{lemma}

\begin{proof} The implied constants in this proof are allowed to depend only on the parameters $d$ and $\lambda_0$. For functions $h_1\in C([T_1,T_2]:\mathcal{G}^{\alpha_0/2,1/3}_{d'}(\T^d\times\R^d))$ and $\in C([T_1,T_2]:\mathcal{G}^{\alpha_0/2,1/3}_{d'}(\T^d))$ we define the operators $\mathcal{A}_ah_1$, $|a|\leq d'$, and $\mathcal{A}h_2$ by multiplication in the Fourier space,
\begin{equation}\label{pam3}
\widehat{\mathcal{A}_ah_1}(t,k,\xi):=A(t,k,\xi)D^a_\xi\widehat{h_1}(t,k,\xi),\qquad \widehat{\mathcal{A}h_2}(t,k):=A(t,k,tk)\widehat{h_2}(t,k,tk).
\end{equation}
We define $\widehat{\rho_i}(t,k):=\widehat{g_i}(t,k,tk)$, $\delta\rho:=\rho_1-\rho_2$, $\delta g:=g_1-g_2$, and then we define $\mathcal{A}\rho_i$, $\mathcal{A}\delta\rho$, $\mathcal{A}_ag_i$, and $\mathcal{A}_a\delta g$, $i\in\{1,2\}$, by the formulas above. We also decompose $\mathcal{N}_i=\mathcal{N}_{i0}-\mathcal{N}_{i1}$,
\begin{equation}\label{pam4}
\begin{split}
&\widehat{\mathcal{N}_{i0}}(t,k):=\widehat{g_i}(T_1,k,tk),\\
&\widehat{\mathcal{N}_{i1}}(t,k):=\frac{1}{(2\pi)^d}\sum_{l\in\Z^d\setminus\{0\}}\int_{T_1}^t\widehat{\rho_i}(s,l)\widehat{g_i}(s,k-l,tk-sl)\frac{(t-s)l\cdot k}{|l|^2}\,ds.
\end{split}
\end{equation}
Using Sobolev embedding (as in Lemma \ref{g*}), for any $k\in\Z^d$, and $t\in[T_1,T_2]$ we have
\begin{equation*}
\begin{split}
e^{2(\lambda_1+4\delta)\langle k,tk\rangle^{1/3}}|\widehat{g_i}(T_1,k,tk)|^2&\lesssim \int_{\R^d}e^{2(\lambda_1+4\delta)\langle k,\xi\rangle^{1/3}}\Big\{\sum_{|a|\leq d'}\big|D^a_\xi\widehat{g_i}(T_1,k,\xi)\big|^2\Big\}\,d\xi,\\
e^{2(\lambda_1+4\delta)\langle k,tk\rangle^{1/3}}|(\widehat{g_1}-\widehat{g_2})(T_1,k,tk)|^2&\lesssim \int_{\R^d}e^{2(\lambda_1+4\delta)\langle k,\xi\rangle^{1/3}}\Big\{\sum_{|a|\leq d'}\big|D^a_\xi(\widehat{g_1}-\widehat{g_2})(T_1,k,\xi)\big|^2\Big\}\,d\xi,
\end{split}
\end{equation*}
where $i\in\{1,2\}$. Moreover, for $k\neq 0$ we have $A(t,k,tk)^2\leq e^{2(\lambda_1+4\delta)\langle k,tk\rangle^{1/3}}e^{-4\delta\langle k,tk\rangle^{1/3}}$. Since $\widehat{g_i}(t,0,0)=0$ and using also the assumption \eqref{boot1}, we have
\begin{equation*}
\sum_{k\in\Z^d\setminus\{0\}}A(t,k,tk)^2|\widehat{\mathcal{N}_{i0}}(t,k)|^2\leq\sum_{k\in\Z^d}e^{2(\lambda_1+4\delta)\langle k,tk\rangle^{1/3}}e^{-4\delta\langle t\rangle^{1/3}}|\widehat{g_i}(T_1,k,tk)|^2\lesssim \overline{\eps}^2e^{-4\delta\langle t\rangle^{1/3}}
\end{equation*}
and similarly
\begin{equation*}
\sum_{k\in\Z^d\setminus\{0\}}A(t,k,tk)^2|(\widehat{\mathcal{N}_{10}}-\widehat{\mathcal{N}_{20}})(t,k)|^2\lesssim\eps_0^2e^{-4\delta\langle t\rangle^{1/3}}.
\end{equation*}
Therefore, for any $i\in\{1,2\}$, $t\in[T_1,T_2]$, and $k\in\Z^d\setminus\{0\}$,
\begin{equation*}
\begin{split}
A(t,k,tk)|\widehat{\mathcal{N}_{i0}}(t,k)||k|^{-\beta}\langle t\rangle^{6d}&\leq C_1\overline{\eps}/2,\\
A(t,k,tk)|(\widehat{\mathcal{N}_{10}}-\widehat{\mathcal{N}_{20}})(t,k)||k|^{-\beta}\langle t\rangle^{6d}&\leq C_1\eps_0/2.
\end{split}
\end{equation*}

Since $\mathcal{N}_i=\mathcal{N}_{i0}-\mathcal{N}_{i1}$, for \eqref{pam2} it remains to prove that for $i\in\{1,2\}$
\begin{equation}\label{pam6}
\begin{split}
\sup_{t\in[T_1,T_2],\,k\in\Z^d\setminus\{0\}}\langle t\rangle^{6d}A(t,k,tk)|\widehat{\mathcal{N}_{i1}}(t,k)||k|^{-\beta}&\leq C_1\overline{\eps}/2,\\
\sup_{t\in[T_1,T_2],\,k\in\Z^d\setminus\{0\}}\langle k,tk\rangle^{-2}\langle t\rangle^{6d}A(t,k,tk)|(\widehat{\mathcal{N}_{11}}-\widehat{\mathcal{N}_{21}})(t,k)||k|^{-\beta}&\leq C_1\eps_0/2.
\end{split}
\end{equation}
Using the definitions \eqref{pam3}--\eqref{pam4} we estimate, for $i\in\{1,2\}$,
\begin{equation*}
\begin{split}
&\langle t\rangle^{6d}A(t,k,tk)|k|^{-\beta}|\widehat{\mathcal{N}_{i1}}(t,k)|\\
&\lesssim \sum_{l\in\Z^d\setminus\{0\}}\int_{T_1}^t\frac{\langle t\rangle^{6d}A(t,k,tk)}{\langle s\rangle^{6d}}\frac{\langle s\rangle^{6d}|\widehat{\mathcal{A}\rho_i}(s,l)|}{A(s,l,sl)|l|^\beta}\frac{|\widehat{\mathcal{A}_0g_i}(s,k-l,tk-sl)|}{A(s,k-l,tk-sl)}\frac{|t-s||k|^{1-\beta}}{|l|^{1-\beta}}\,ds\\
&\lesssim \sum_{l\in\Z^d\setminus\{0\}}\int_{T_1}^t\frac{\langle s\rangle^{6d}|\widehat{\mathcal{A}\rho_i}(s,l)|}{|l|^\beta}G_{\ast 0i}(s,k-l)\cdot\frac{|t-s||k|^{1-\beta}\langle t\rangle^{6d}}{|l|^{1-\beta}\langle s\rangle^{6d}}\frac{A(t,k,tk)}{A(s,l,sl)A(s,k-l,tk-sl)}\,ds,
\end{split}
\end{equation*}
where $G_{\ast 0i}(s,n):=\sup_{\xi\in\R}|\widehat{\mathcal{A}_0g_i}(s,n,\xi)|$. The assumptions \eqref{boot2} and the bounds \eqref{gu11} show that $\big\||l|^{-\beta}\langle s\rangle^{6d}\widehat{\mathcal{A}\rho_i}(s,l)\big\|_{L^\infty_{s,l}}\leq 2\overline{C}\overline{\eps}$ and $\big\|G_{\ast 0i}(s,l)\big\|_{L^\infty_sL^2_l}\leq 2C_0\overline{C}\overline{\eps}$. Moreover, in view \eqref{gu6}
\begin{equation}\label{pam6.2}
\begin{split}
&\frac{\langle k,tk\rangle^{-p}A(t,k,tk)}{\langle l,sl\rangle^{-p}A(s,l,sl)\langle k-l,tk-sl\rangle^{-p}A(s,k-l,tk-sl)}\\
&\qquad\leq \frac{A(s,k,tk)}{A(s,l,sl)A(s,k-l,tk-sl)}\frac{\langle l,sl\rangle^p\langle k-l,tk-sl\rangle^p}{\langle k,tk\rangle^p}\frac{A(t,k,tk)}{A(s,k,tk)}\\
&\qquad\lesssim e^{-8\delta\min(\langle l,sl\rangle,\langle k-l,tk-sl\rangle)^{1/3}}\frac{A(t,k,tk)}{A(s,k,tk)},
\end{split}
\end{equation}
for $p\in[0,2]$. Since $\overline{C}^8\leq\overline{\eps}^{-1}$, to prove the bounds in the first line of \eqref{pam6} it suffices to prove that for any $t\in[T_1,T_2]$ and $k\in\Z^d\setminus\{0\}$ we have
\begin{equation}\label{pam7}
\sum_{l\in\Z^d\setminus\{0\}}\int_{T_1}^t\frac{|t-s||k|^{1-\beta}\langle t\rangle^{6d}}{|l|^{1-\beta}\langle s\rangle^{6d}}\frac{A(t,k,tk)}{A(s,k,tk)}\big[e^{-8\delta\langle l,sl\rangle^{1/3}}+e^{-8\delta\langle k-l,tk-sl\rangle^{1/3}}\big]\,ds\lesssim 1.
\end{equation}
The bounds in the second line of \eqref{pam6} also follow from \eqref{pam7}, using the bounds \eqref{pam6.2} with $p=2$.

We prove the bounds \eqref{pam7} in three steps.

{\bf{Step 1.}} We show first that for any $t\in[0,\infty)$ and $k\in\Z^d\setminus\{0\}$
\begin{equation}\label{pam8}
\begin{split}
\sum_{l\in\Z^d\setminus\{0\}}\int_0^t\frac{|t-s||k|^{1-\beta}\langle t\rangle^{6d}}{|l|^{1-\beta}\langle s\rangle^{6d}}\frac{A(t,k,tk)}{A(s,k,tk)}e^{-\delta(s^{1/3}+|l|^{1/3})}\,ds\lesssim 1.
\end{split}
\end{equation}
For this we use \eqref{gu4}, so if $0\leq s\leq t$
\begin{equation}\label{pam8.5}
\frac{A(t,k,tk)}{A(s,k,tk)}=e^{\lambda(t,|k,tk|)-\lambda(s,|k,tk|)}\leq e^{-\delta^2|t-s|(1+t)^{-\delta-1}|k|^{1/3}(1+t^2)^{1/6}}\leq e^{-\delta^2|t-s||k|^{1/3}(1+t)^{-2/3-\delta}/2}.
\end{equation}
Therefore, if $\beta\in[1/2,1]$, $k,l\in\Z^d\setminus\{0\}$, and $s\in [t/2,t]$  then
\begin{equation*}
\begin{split}
\frac{|t-s||k|^{1-\beta}\langle t\rangle^{6d}}{|l|^{1-\beta}\langle s\rangle^{6d}}\frac{A(t,k,tk)}{A(s,k,tk)}e^{-\delta(s^{1/3}+|l|^{1/3})}&\lesssim e^{-\delta|l|^{1/3}}e^{-\delta t^{1/3}/2}|k|^{1-\beta}|t-s|e^{-\delta^2|t-s|(1+t)^{-1}|k|^{1/3}/2}\\
&\lesssim e^{-\delta|l|^{1/3}}e^{-\delta t^{1/3}/4}|k|^{1/2}|t-s|(1+|t-s||k|^{1/3})^{-2}\\
&\lesssim e^{-\delta|l|^{1/3}}e^{-\delta t^{1/3}/4}|t-s|^{-1/2}.
\end{split}
\end{equation*}
Moreover, if $\beta\in[1/2,1]$, $k,l\in\Z^d\setminus\{0\}$, $s\leq t/2$, and $t\leq 1$ then
\begin{equation*}
\begin{split}
\frac{|t-s||k|^{1-\beta}\langle t\rangle^{6d}}{|l|^{1-\beta}\langle s\rangle^{6d}}\frac{A(t,k,tk)}{A(s,k,tk)}e^{-\delta(s^{1/3}+|l|^{1/3})}&\lesssim |t-s||k|^{1-\beta}e^{-\delta^2|t-s||k|^{1/3}/4}e^{-\delta|l|^{1/3}}\\
&\lesssim e^{-\delta|l|^{1/3}}|t-s|^{-1/2}.
\end{split}
\end{equation*}
Finally, if $\beta\in[1/2,1]$, $k,l\in\Z^d\setminus\{0\}$, $s\leq t/2$, and $t\geq 1$ then
\begin{equation*}
\begin{split}
\frac{|t-s||k|^{1-\beta}\langle t\rangle^{6d}}{|l|^{1-\beta}\langle s\rangle^{6d}}\frac{A(t,k,tk)}{A(s,k,tk)}e^{-\delta(s^{1/3}+|l|^{1/3})}&\lesssim \langle t\rangle^{7d}|k|^{1-\beta}e^{-\delta^2t^{1/4}|k|^{1/3}/8}e^{-\delta|l|^{1/3}}\lesssim e^{-\delta|l|^{1/3}}\langle t\rangle^{-4}.
\end{split}
\end{equation*}
Therefore, in all cases, 
\begin{equation}\label{pam9}
\frac{|t-s||k|^{1-\beta}\langle t\rangle^{6d}}{|l|^{1-\beta}\langle s\rangle^{6d}}\frac{A(t,k,tk)}{A(s,k,tk)}e^{-\delta(s^{1/3}+|l|^{1/3})}\lesssim e^{-\delta|l|^{1/3}}\langle t\rangle^{-3}|t-s|^{-1/2},
\end{equation}
and the desired bounds \eqref{pam8} follow.

{\bf{Step 2.}} We show now that for any $t\in[0,\infty)$ and $k\in\Z^d\setminus\{0\}$ we have
\begin{equation}\label{pam10}
\sum_{l\in\Z^d\setminus\{0\}}\int_0^t\mathbf{1}_{\mathcal{R}}(l,s)\frac{|t-s||k|^{1-\beta}\langle t\rangle^{6d}}{|l|^{1-\beta}\langle s\rangle^{6d}}\frac{A(t,k,tk)}{A(s,k,tk)}e^{-\delta(|k-l|^{1/3}+|tk-sl|^{1/3})}\,ds\lesssim 1,
\end{equation}
where
\begin{equation}\label{pam11}
\mathcal{R}:=\big\{(l,s)\in(\Z^d\setminus\{0\})\times[0,t]:\,s\leq 0.99(t+1)\,\,\text{ or }\,\, |k-l|\geq |l|/10\big\}.
\end{equation}

For this we notice that if $k,l\in\Z^d\setminus\{0\}$ then
\begin{equation}\label{pam12}
\frac{|k|^{1-\beta}}{|l|^{1-\beta}}e^{-\delta|k-l|^{1/3}}\lesssim e^{-\delta|k-l|^{1/3}/2}.
\end{equation}
Recall also the bounds \eqref{pam8.5}. Therefore, if $k,l\in\Z^d\setminus\{0\}$ and $s\leq \min\{t,0.99 (t+1)\}$ then
\begin{equation}\label{pam13}
\frac{|t-s|\langle t\rangle^{6d}}{\langle s\rangle^{6d}}\frac{A(t,k,tk)}{A(s,k,tk)}e^{-\delta |tk-sl|^{1/3}}\lesssim e^{-\delta^4(1+t)^{1/4}}.
\end{equation}
Moreover, if $k,l\in\Z^d\setminus\{0\}$, $s\in[0.99 (t+1),t]$, and $|k-l|\geq |l|/10$ then $|tk-sl|=|t(k-l)+(t-s)l|\geq t|k-l|/2$, and the bounds \eqref{pam13} follow in this case as well. The desired bounds \eqref{pam10} follow from \eqref{pam12}--\eqref{pam13}.

{\bf{Step 3.}} Finally, we show that
\begin{equation}\label{pam15}
\sum_{l\in\Z^d\setminus\{0\}}\int_0^t(1-\mathbf{1}_{\mathcal{R}}(l,s))\frac{|t-s||k|^{1-\beta}\langle t\rangle^{6d}}{|l|^{1-\beta}\langle s\rangle^{6d}}\frac{A(t,k,tk)}{A(s,k,tk)}e^{-\delta(|k-l|^{1/3}+|tk-sl|^{1/3})}\,ds\lesssim 1,
\end{equation}
where $\mathcal{R}$ is as in \eqref{pam11}. We make the change of variables $l=k-a$, so it suffices to prove that if $k\in\Z^d\setminus\{0\}$ and $t\in[10,\infty)$ then
\begin{equation}\label{pam16}
\begin{split}
\sum_{a\in\Z^d,\,|a|\leq |k|/8}\int_{0.99t}^{t}\frac{|t-s|A(t,k,tk)}{A(s,k,tk)}e^{-\delta(|a|^{1/3}+|sa+(t-s)k|^{1/3})}\,ds\lesssim 1.
\end{split}
\end{equation}

The bounds \eqref{pam8.5} are not sufficient in this case. We examine the identity \eqref{gu4} and use the second term in the right-hand side to see that
\begin{equation*}
\begin{split}
\lambda(s,|k,tk|)-\lambda(t,|k,tk|)&\geq \delta^2|t-s|(1+t\langle k,tk\rangle^{-2/3})^{-1-\delta}\langle k,tk\rangle^{-1/3}\\
&\geq\delta^3|t-s| t^{-1/3}|k|^{-1/3}(1+t^{1/3}|k|^{-2/3})^{-1-\delta},
\end{split}
\end{equation*}
if $t\geq 10$, $s\in[0.99t,t]$, and $k\in\Z^d\setminus\{0\}$. Therefore
\begin{equation}\label{pam17}
\frac{A(t,k,tk)}{A(s,k,tk)}\leq e^{-\delta^4|t-s|t^{-1/3}|k|^{-1/3}(1+t|k|^{-2})^{-1/3-\delta}}.
\end{equation}
Since $|sa+(t-s)k|\geq \big|s|a|-(t-s)|k|\big|$, for \eqref{pam16} it suffices to show that if $a\in\Z^d$, $k\in\Z^d\setminus\{0\}$, $t\in[10,T]$, and $|a|\leq|k|/8$ then
\begin{equation}\label{pam18}
\begin{split}
&\int_{0.99t}^{t}K_a(t,k,s)\,ds\lesssim 1,\\
&K_a(t,k,s):=|t-s|e^{-\delta^4|t-s|t^{-1/3}|k|^{-1/3}(1+t|k|^{-2})^{-1/3-\delta}}e^{-\delta|s|a|-(t-s)|k||^{1/3}}e^{-\delta |a|^{1/3}/2}.
\end{split}
\end{equation}

The bounds \eqref{pam18} are easy in the case $a=0$, since $|K_0(t,k,s)|\lesssim |t-s|e^{-\delta|t-s|^{1/3}}\lesssim \langle t-s\rangle^{-4}$. On the other hand, if $a\in\Z^d\setminus\{0\}$ then
\begin{equation*}
|K_a(t,k,s)|\lesssim |t-s|e^{-\delta|t-s|^{1/3}|k|^{1/3}/2}\lesssim \langle t-s\rangle^{-4}
\end{equation*}
if $(t-s)|k|\notin[t|a|/2,2t|a|]$, and
\begin{equation*}
\begin{split}
|K_a(t,k,s)|&\lesssim t|k|^{-1}e^{-\delta^4(t|k|^{-2})^{2/3}(1+t|k|^{-2})^{-1/3-\delta}/4}e^{-\delta|s|a|-(t-s)|k||^{1/3}}e^{-\delta |a|^{1/3}/4}\\
&\lesssim |k|e^{-\delta|s|a|-(t-s)|k||^{1/3}}e^{-\delta |a|^{1/3}/4}
\end{split}
\end{equation*}
if $(t-s)|k|\in[t|a|/2,2t|a|]$. The desired bounds \eqref{pam18} follow in this case as well.
\end{proof}

We can now prove our main estimates on the function $\rho_1$, $\rho_2$, and $\rho_1-\rho_2$.

\begin{proposition}\label{MainBootstrapIm1} With the assumptions and notations of Proposition \ref{MainBootstrap}, there is a constant $C_2=C_2(d,\lambda_0,\vartheta)\geq 1$ such that for any $t\in[T_1,T_2]$ and $i\in\{1,2\}$
\begin{equation}\label{pam20}
\mathcal{Z}_{g_i}^{0}(t)\langle t\rangle^{6d}\leq C_2\overline{\eps}\qquad\text{ and }\qquad \mathcal{Z}_{g_1-g_2}^{2}(t)\langle t\rangle^{6d}\leq C_2\eps_0.
\end{equation}
As a consequence, if the constant $\overline{C}$ is chosen sufficiently large, $\overline{C}\geq C_2$, then there is a constant $C'_0=C'_0(d)$ such that
\begin{equation}\label{pam21}
\begin{split}
\big\|A(t,k,tk)\widehat{\rho_i}(t,k)|k|^{-\beta}\|_{L^2_k}&\leq C'_0(\overline{C}^dC_2)^{1/(d+1)}\overline{\epsilon}\langle t\rangle^{-3},\\
\big\|\langle k,tk\rangle^{-2} A(t,k,tk)(\widehat{\rho_1-\rho_2})(t,k)|k|^{-\beta}\|_{L^2_k}&\leq C'_0(\overline{C}^dC_2)^{1/(d+1)}\epsilon_0\langle t\rangle^{-3}
\end{split}
\end{equation}
for any $t\in[T_1,T_2]$ and $i\in\{1,2\}$.
\end{proposition}

\begin{proof} We remark that the $L^2$ bounds \eqref{pam21} are needed later, in the proof of Lemma \ref{Lem1ner} below. 

To prove the bounds \eqref{pam20} we use the representation formula \eqref{pam27.1}. For any $t\in [T_1,T_2]$, $k\in\Z^d\setminus\{0\}$, and $i\in\{1,2\}$ we estimate
\begin{equation*}
\begin{split}
\langle t\rangle^{6d}A(t,k,tk)&|k|^{-\beta}\big|\widehat{\rho_i}(t,k)\big|\leq \langle t\rangle^{6d}A(t,k,tk)|k|^{-\beta}\big|\widehat{\mathcal{N}_i}(t,k)\big|\\
&+\int_0^t\big[\langle s\rangle^{6d}A(s,k,sk)|k|^{-\beta}\big|\widehat{\mathcal{N}_i}(s,k)\big|\big]\cdot \big|\widehat{G}(t-s,k)\big|\frac{\langle t\rangle^{6d}A(t,k,tk)}{\langle s\rangle^{6d}A(s,k,sk)}\,ds.
\end{split}
\end{equation*}
Using Lemma \ref{LemN}, the bounds \eqref{gu6}, and the bounds on $\widehat{G}$ in \eqref{pam27.2} we further estimate
\begin{equation*}
\begin{split}
\langle t\rangle^{6d}A(t,k,tk)|k|^{-\beta}&\big|\widehat{\rho_i}(t,k)\big|\leq C_1\overline{\epsilon}+C_1\overline{\epsilon}\int_0^t\big|\widehat{G}(t-s,k)\big|\frac{\langle t\rangle^{6d}A(s,k,tk)}{\langle s\rangle^{6d}A(s,k,sk)}\,ds\\
&\lesssim_{d,\lambda_0,\vartheta}\overline{\epsilon}\Big[1+\int_0^te^{-0.95\lambda_0\langle (t-s)k\rangle^{1/3}}\langle t-s\rangle^{6d}A(s,0,(t-s)k)\,ds\Big].
\end{split}
\end{equation*}
The desired bounds on $\mathcal{Z}_{g_i}^{0}(t)\langle t\rangle^{6d}$ in \eqref{pam20} follow since $A(s,0,(t-s)k)\leq e^{(\lambda_1+2\delta)\langle (t-s)k\rangle^{1/3}}$. The bounds on $\mathcal{Z}_{g_1-g_2}^{2}(t)\langle t\rangle^{6d}$ follow similarly.

To prove the bounds \eqref{pam21} we recall that $\beta=1/2$. Using Lemma \ref{g*} and the bootstrap assumption \eqref{boot2}, for any $t\in[T_1,T_2]$ and $i\in\{1,2\}$ we have
\begin{equation*}
\sum_{k\in\Z^d,\,|k|\geq D}|A(t,k,tk)\widehat{\rho_i}(t,k)|^2|k|^{-1}\lesssim_d \overline{C}^2\overline{\epsilon}^2D^{-1},
\end{equation*}
where $D:=\langle t\rangle^{6}(\overline{C}/C_2)^{2/(d+1)}$ is a suitable constant. On the other hand, using \eqref{pam20}, 
\begin{equation*}
\sum_{k\in\Z^d\setminus\{0\},\,|k|\leq D}|A(t,k,tk)\widehat{\rho_i}(t,k)|^2|k|^{-1}\leq \sum_{k\in\Z^d,\,|k|\leq D}C_2^2\overline{\epsilon}^2\langle t\rangle^{-12d}\lesssim_dC_2^2\overline{\epsilon}^2\langle t\rangle^{-12d}D^d.
\end{equation*}
The desired bounds in the first line of \eqref{pam21} follow from these two inequalities. The bounds in the second line follow in a similar way.
\end{proof}

\subsection{Improved control on energy functionals}\label{ham1}

In this subsection we prove the following:

\begin{proposition}\label{MainBootstrapIm2}
With the assumptions and notations of Proposition \ref{MainBootstrap}, for any $t\in[T_1,T_2]$ and $i\in\{1,2\}$ we have
\begin{equation}\label{ham2}
\mathcal{E}^0_{g_i}(t)\leq\overline{C}^2\overline{\eps}^2/4\qquad\text{ and }\qquad\mathcal{E}^2_{g_1-g_2}(t)\leq\overline{C}^2\eps_0^2/4.
\end{equation}
\end{proposition}

To justify formally the calculations, for any $h\in X_{[T_1,T_2]}$ we define the mollified energies
\begin{equation}\label{ham2.1}
\begin{split}
&\mathcal{E}^{p,L}_h(t):=\sum_{k\in \mathbb{Z}^d}\int_{\R^d}\langle k,\xi\rangle^{-2p}A_L(t,k,\xi)^2\Big\{\sum_{|a|\leq d'}\big|D^a_\xi\widehat{h}(t,k,\xi)\big|^2\Big\}\,d\xi,\\
&A_L(t,k,\xi):=(1+2^{-L}\langle k,\xi\rangle)^{-4}A(t,k,\xi),
\end{split}
\end{equation}
where $p\in[0,2]$ and $L\geq 4$. We calculate
\begin{equation*}
\begin{split}
\frac{d}{dt}\mathcal{E}^{0,L}_{g_i}(t)&=2\sum_{k\in \mathbb{Z}^d}\int_{\R^d}A_L(t,k,\xi)\dot{A_L}(t,k,\xi)\Big\{\sum_{|a|\leq d'}\big|D^a_\xi\widehat{g_i}(t,k,\xi)\big|^2\Big\}\,d\xi\\
&+2\Re\Big\{\sum_{k\in \mathbb{Z}^d}\int_{\R^d}A_L(t,k,\xi)^2\Big\{\sum_{|a|\leq d'}D^a_\xi(\partial_t\widehat{g_i})(t,k,\xi)\overline{D^a_\xi\widehat{g_i}(t,k,\xi)}\Big\}\,d\xi\Big\},
\end{split}
\end{equation*}
where $\dot{A_L}=\partial_tA_L$. Using the formula \eqref{NVP4} for $\partial_t\widehat{g_i}$ we obtain
\begin{equation}\label{ner1}
\frac{d}{dt}\mathcal{E}^{0,L}_{g_i}(t)=I_i(t)+II_i(t)+III_i(t)+IV_i(t),
\end{equation}
where
\begin{equation}\label{ner2}
I_i(t):=2\sum_{|a|\leq d'}\sum_{k\in \mathbb{Z}^d}\int_{\R^d}A_L(t,k,\xi)\dot{A_L}(t,k,\xi)\big|D^a_\xi\widehat{g_i}(t,k,\xi)\big|^2\,d\xi,
\end{equation}
\begin{equation}\label{ner3}
\begin{split}
II_i(t)&:=-2\Re\Big\{\sum_{|a|\leq d'}\sum_{k\in \mathbb{Z}^d\setminus\{0\}}\int_{\R^d}A_L(t,k,\xi)^2\overline{D^a_\xi\widehat{g_i}(t,k,\xi)}\\
&\qquad\qquad\times D^a_\xi\Big[\widehat{\rho_i}(t,k)\widehat{M_0}(\xi-tk)\frac{k\cdot (\xi-tk)}{|k|^2}\Big]\,d\xi\Big\},
\end{split}
\end{equation}
\begin{equation}\label{ner4}
\begin{split}
III_i(t)&:=-\frac{2}{(2\pi)^d}\Re\Big\{\sum_{|a|\leq d'}\sum_{k\in \mathbb{Z}^d}\int_{\R^d}A_L(t,k,\xi)^2\overline{D^a_\xi\widehat{g_i}(t,k,\xi)}\\
&\qquad\qquad\qquad\times \Big[\sum_{l\in\Z^d\setminus\{0\}}\widehat{\rho_i}(t,l)D^a_\xi\widehat{g_i}(t,k-l,\xi-tl)\frac{l\cdot(\xi-tk)}{|l|^2}\Big]\,d\xi\Big\},
\end{split}
\end{equation}
\begin{equation}\label{ner5}
\begin{split}
IV_i(t)&:=-\frac{2}{(2\pi)^d}\Re\Big\{\sum_{|a|\leq d'}\sum_{|a'|=1,\,a'\leq a}\sum_{k\in \mathbb{Z}^d}\int_{\R^d}A_L(t,k,\xi)^2\overline{D^a_\xi\widehat{g_i}(t,k,\xi)}\\
&\qquad\qquad\qquad\times \Big[\sum_{l\in\Z^d\setminus\{0\}}\widehat{\rho_i}(t,l)(D^{a-a'}_\xi\widehat{g_i})(t,k-l,\xi-tl)\frac{D^{a'}_\xi[l\cdot(\xi-tk)]}{|l|^2}\Big]\,d\xi\Big\}.
\end{split}
\end{equation}

Notice that $I_i(t)$ are good terms, $I_i(t)\leq 0$, since $\dot{A_L}(t,k,\xi)\leq 0$. Using \eqref{ner1} we write
\begin{equation}\label{ner6}
\begin{split}
\mathcal{E}^{0,L}_{g_i}(t)&=\mathcal{E}^{0,L}_{g_i}(T_1)+\int_{T_1}^t\big[I_i(s)+II_i(s)+III_i(s)+IV_i(s)\big]\,ds\\
&\leq\mathcal{E}^0_{g_i}(T_1)+\int_{T_1}^t|II_i(s)|\,ds+\int_{T_1}^t|IV_i(s)|\,ds+\int_{T_1}^t\big[|III_i(s)|-|I_i(s)|\big]\,ds.
\end{split}
\end{equation}
The bounds on $\mathcal{E}^0_{g_i}(t)$ in \eqref{ham2} follow from Lemmas \ref{Lem1ner}, \ref{Lem2ner}, \ref{Lem3ner} below, by letting $L\to\infty$.

Similarly, we calculate
\begin{equation}\label{ner1.3}
\frac{d}{dt}\mathcal{E}^{2,L}_{g_1-g_2}(t)=\delta I(t)+\delta II(t)+\delta III(t)+\delta IV(t),
\end{equation}
where
\begin{equation}\label{ner2.3}
\delta I(t):=2\sum_{|a|\leq d'}\sum_{k\in \mathbb{Z}^d}\int_{\R^d}\langle k,\xi\rangle^{-4}A_L(t,k,\xi)\dot{A_L}(t,k,\xi)\big|D^a_\xi(\widehat{g_1}-\widehat{g_2})(t,k,\xi)\big|^2\,d\xi,
\end{equation}
\begin{equation}\label{ner3.3}
\begin{split}
\delta II(t)&:=-2\Re\Big\{\sum_{|a|\leq d'}\sum_{k\in \mathbb{Z}^d\setminus\{0\}}\int_{\R^d}\langle k,\xi\rangle^{-4}A_L(t,k,\xi)^2\overline{D^a_\xi(\widehat{g_1}-\widehat{g_2})(t,k,\xi)}\\
&\qquad\qquad\times D^a_\xi\Big[(\widehat{\rho_1}-\widehat{\rho_2})(t,k)\widehat{M_0}(\xi-tk)\frac{k\cdot (\xi-tk)}{|k|^2}\Big]\,d\xi\Big\},
\end{split}
\end{equation}
\begin{equation}\label{ner4.3}
\begin{split}
&\delta III(t):=-\frac{2}{(2\pi)^d}\Re\Big\{\sum_{|a|\leq d'}\sum_{k\in \mathbb{Z}^d}\int_{\R^d}\langle k,\xi\rangle^{-4}A_L(t,k,\xi)^2\overline{D^a_\xi(\widehat{g_1}-\widehat{g_2})(t,k,\xi)}\\
&\quad\times \Big[\sum_{l\in\Z^d\setminus\{0\}}[\widehat{\rho_1}(t,l)D^a_\xi\widehat{g_1}(t,k-l,\xi-tl)-\widehat{\rho_2}(t,l)D^a_\xi\widehat{g_2}(t,k-l,\xi-tl)]\frac{l\cdot(\xi-tk)}{|l|^2}\Big]\,d\xi\Big\},
\end{split}
\end{equation}
\begin{equation}\label{ner5.3}
\begin{split}
&\delta IV(t):=-\frac{2}{(2\pi)^d}\Re\Big\{\sum_{|a|\leq d'}\sum_{|a'|=1,\,a'\leq a}\sum_{k\in \mathbb{Z}^d}\int_{\R^d}\langle k,\xi\rangle^{-4}A_L(t,k,\xi)^2\overline{D^a_\xi(\widehat{g_1}-\widehat{g_2})(t,k,\xi)}\\
&\quad\times \Big[\sum_{l\in\Z^d\setminus\{0\}}[\widehat{\rho_1}(t,l)(D^{a-a'}_\xi\widehat{g_1})(t,k-l,\xi-tl)-\widehat{\rho_2}(t,l)(D^{a-a'}_\xi\widehat{g_2})(t,k-l,\xi-tl)]\\
&\qquad\qquad\qquad\times\frac{D^{a'}_\xi[l\cdot(\xi-tk)]}{|l|^2}\Big]\,d\xi\Big\}.
\end{split}
\end{equation}
As before, we notice that $\delta I(t)\leq 0$, so
\begin{equation}\label{ner6.3}
\mathcal{E}^{2,L}_{g_1-g_2}(t)\leq\mathcal{E}^2_{g_1-g_2}(T_1)+\int_{T_1}^t|\delta II(s)|\,ds+\int_{T_1}^t|\delta IV(s)|\,ds+\int_{T_1}^t\big[|\delta III(s)|-|\delta I(s)|\big]\,ds.
\end{equation}
The bounds on $\mathcal{E}^2_{g_1-g_2}(t)$ in \eqref{ham2} follow from Lemmas \ref{Lem1ner}, \ref{Lem2ner}, \ref{Lem3ner} below by letting $L\to\infty$.

\begin{lemma}\label{Lem1ner}
For any $t\in[T_1,T_2]$, $L\geq 4$, and $i\in\{1,2\}$ we have
\begin{equation}\label{ner7}
\int_{T_1}^t\big|II_i(s)\big|\,ds\leq\overline{C}^2\overline{\eps}^2/8\qquad\text{ and }\qquad \int_{T_1}^t\big|\delta II(s)\big|\,ds\leq\overline{C}^2\eps_0^2/8.
\end{equation}
\end{lemma}

\begin{proof} Let
\begin{equation}\label{ner8}
M_0^\ast(\eta):=\sum_{|a|\leq d',\, j\in\{1,\ldots,d\}}\big|D^a_\eta\big[\widehat{M_0}(\eta)\eta_j\big]\big|.
\end{equation}
It follows from \eqref{ner3} that
\begin{equation}\label{ner9}
\begin{split}
|II_i(s)|\leq 2\sum_{|a|\leq d'}\sum_{k\in \mathbb{Z}^d\setminus\{0\}}\int_{\R^d}&A(s,k,\xi)|D^a_\xi\widehat{g_i}(s,k,\xi)|\\
&\times\frac{A(s,k,sk)|\widehat{\rho_i}(s,k)|}{|k|}\frac{A(s,k,\xi)}{A(s,k,sk)}M_0^\ast(\xi-sk)\,d\xi.
\end{split}
\end{equation}
Moreover, using \eqref{gu6},
\begin{equation*}
\frac{A(s,k,\xi)}{A(s,k,sk)}M_0^\ast(\xi-sk)\leq A(s,0,\xi-sk)M_0^\ast(\xi-sk)\leq e^{0.95\lambda_0\langle\xi-sk\rangle^{1/3}}M_0^\ast(\xi-sk).
\end{equation*}
It follows from the assumption \eqref{M01} and the definition \eqref{ner8} that
\begin{equation}\label{ner9.5}
\big\|e^{0.95\lambda_0\langle\eta\rangle^{1/3}}M_0^\ast(\eta)\big\|_{L^2_\eta}\lesssim_{d,\lambda_0,\vartheta}1,
\end{equation}
with an implied constant that depends on $d,\lambda_0,\vartheta$. Recall the definition \eqref{pam3} and let
\begin{equation}\label{ner10}
G_{\ast i}(t,k):=\sum_{|a|\leq d'}\big\|\widehat{\mathcal{A}_ag_i}(t,k,\xi)\big\|_{L^2_\xi},\qquad\delta G_{\ast}(t,k):=\sum_{|a|\leq d'}\big\|\langle k,\xi\rangle^{-2}\widehat{\mathcal{A}_a\delta g}(t,k,\xi)\big\|_{L^2_\xi},
\end{equation}
for $t\in[T_1,T_2]$ and $k\in\Z^d$. It follows from \eqref{ner9} and the Cauchy inequality in $\xi$ that
\begin{equation}\label{ner11}
|II_i(s)|\lesssim_{d,\lambda_0,\vartheta}\sum_{k\in \mathbb{Z}^d\setminus\{0\}}G_{\ast i}(s,k)\frac{A(s,k,sk)|\widehat{\rho_i}(s,k)|}{|k|},
\end{equation}
 It follows from the bootstrap assumption \eqref{boot2} that $\|G_{\ast i}\|_{L^\infty_sL^2_k}\lesssim_d \overline{C}\overline{\epsilon}$. Using also \eqref{pam21} and the Cauchy inequality we have
\begin{equation*}
|II_i(s)|\lesssim_{d,\lambda_0,\vartheta}\overline{C}\overline{\eps}\cdot (\overline{C}^dC_2)^{1/(d+1)}\overline{\epsilon}\langle s\rangle^{-3}.
\end{equation*}
The bounds for $|II_i(s)|$ in \eqref{ner7} follow if the constant $\overline{C}$ is taken sufficiently large relative to $C_2$. The bounds for $|\delta II(s)|$ in \eqref{ner7} follow in a similar way.
\end{proof}

\begin{lemma}\label{Lem2ner}
For any $t\in[T_1,T_2]$, $L\geq 4$, and $i\in\{1,2\}$ we have
\begin{equation}\label{ner20}
\int_{T_1}^t\big|IV_i(s)\big|\,ds\leq\overline{\eps}^2\qquad\text{ and }\qquad \int_{T_1}^t\big|\delta IV(s)\big|\,ds\leq\eps_0^2.
\end{equation}
\end{lemma}

\begin{proof} The implied constants in this lemma may depend only on $d$ and $\lambda_0$. We provide all the details only for the slightly harder bounds on the difference of solutions. Notice that
\begin{equation}\label{ner20.1}
\begin{split}
&|\delta IV(s)|\lesssim\sum_{|a|\leq d'}\sum_{|a'|=1,\,a'\leq a}\sum_{k\in \mathbb{Z}^d}\sum_{l\in\Z^d\setminus\{0\}}\int_{\R^d}\langle k,\xi\rangle^{-4}A(s,k,\xi)^2|D^a_\xi\widehat{\delta g}(s,k,\xi)| |l|^{-1}\\
&\qquad\times\big\{|\widehat{\delta\rho}(s,l)||D^{a-a'}_\xi\widehat{g_1}(s,k-l,\xi-sl)|+|\widehat{\rho_2}(s,l)||D^{a-a'}_\xi\widehat{\delta g}(s,k-l,\xi-sl)|\big\}\,d\xi,
\end{split}
\end{equation}
where $\delta g=g_1-g_2$ and $\delta\rho=\rho_1-\rho_2$. We define the functions $\mathcal{A}\delta\rho$ and $\mathcal{A}_a\delta g$ as in \eqref{pam3}, for $s\in[T_1,T_2]$, $k\in\Z^d$, $\xi\in\R^d$, and $|a|\leq d'$. It follows from \eqref{ner20.1} that
\begin{equation*}
\begin{split}
&|\delta IV(s)|\lesssim\sum_{|a|,|a'|\leq d'}\sum_{k\in \mathbb{Z}^d}\sum_{l\in\Z^d\setminus\{0\}}\int_{\R^d}\langle k,\xi\rangle^{-2}|\widehat{\mathcal{A}_a\delta g}(s,k,\xi)| \frac{\langle k,\xi\rangle^{-2}A(s,k,\xi)}{A(s,l,sl)A(s,k-l,\xi-sl)}\\
&\qquad\times\big\{|l|^{-1}|\widehat{\mathcal{A}\delta\rho}(s,l)||\widehat{\mathcal{A}_{a'}g_1}(s,k-l,\xi-sl)|+|l|^{-1}|\widehat{\mathcal{A}\rho_2}(s,l)||\widehat{\mathcal{A}_{a'}\delta g}(s,k-l,\xi-sl)|\big\}\,d\xi.
\end{split}
\end{equation*}
In view of \eqref{gu6},
\begin{equation*}
\frac{\langle k,\xi\rangle^{-2}A(s,k,\xi)}{\langle l,sl\rangle^{-2}A(s,l,sl)\langle k-l,\xi-sl\rangle^{-2}A(s,k-l,\xi-sl)}\lesssim e^{-8\delta\langle l,sl\rangle^{1/3}}+e^{-8\delta\langle k-l,\xi-sl\rangle^{1/3}}.
\end{equation*}
Therefore for any $s\in[T_1,T_2]$ we can estimate
\begin{equation}\label{ner21}
|\delta IV(s)|\lesssim \delta IV_1(s)+\delta IV_2(s),
\end{equation}
where
\begin{equation}\label{ner21.1}
\begin{split}
&\delta IV_1(s):=\sum_{|a|,|a'|\leq d'}\sum_{k\in \mathbb{Z}^d}\sum_{l\in\Z^d\setminus\{0\}}\int_{\R^d}\langle k,\xi\rangle^{-2}|\widehat{\mathcal{A}_a\delta g}(s,k,\xi)|e^{-\delta\langle l,sl\rangle^{1/3}}\\
&\quad\times\frac{|\widehat{\mathcal{A}\delta\rho}(s,l)||\widehat{\mathcal{A}_{a'}g_1}(s,k-l,\xi-sl)|+|\widehat{\mathcal{A}\rho_2}(s,l)||\widehat{\mathcal{A}_{a'}\delta g}(s,k-l,\xi-sl)|}{|l|\langle l,sl\rangle^2\langle k-l,\xi-sl\rangle^2}\,d\xi,
\end{split}
\end{equation}
\begin{equation}\label{ner21.2}
\begin{split}
&\delta IV_2(s):=\sum_{|a|,|a'|\leq d'}\sum_{k\in \mathbb{Z}^d}\sum_{l\in\Z^d\setminus\{0\}}\int_{\R^d}\langle k,\xi\rangle^{-2}|\widehat{\mathcal{A}_a\delta g}(s,k,\xi)|e^{-\delta\langle k-l,\xi-sl\rangle^{1/3}}\\
&\quad\times\frac{|\widehat{\mathcal{A}\delta\rho}(s,l)||\widehat{\mathcal{A}_{a'}g_1}(s,k-l,\xi-sl)|+|\widehat{\mathcal{A}\rho_2}(s,l)||\widehat{\mathcal{A}_{a'}\delta g}(s,k-l,\xi-sl)|}{|l|\langle l,sl\rangle^2\langle k-l,\xi-sl\rangle^2}\,d\xi.
\end{split}
\end{equation}

It follows from \eqref{boot2} that $\|\widehat{\mathcal{A}_ag_1}(t,k,\xi)\|_{L^2_{k,\xi}}\leq 2\overline{C}\overline{\epsilon}$ and $\|\langle k,\xi\rangle^{-2}\widehat{\mathcal{A}_a\delta g}(t,k,\xi)\|_{L^2_{k,\xi}}\leq 2\overline{C}\epsilon_0$ for any $t\in[T_1,T_2]$ and multi-index $a$ with $|a|\leq d'$. Therefore, using the Cauchy inequality in $k,\xi$ and recalling the definition \eqref{rec1},
\begin{equation*}
\delta IV_1(s)\lesssim(\overline{C}\epsilon_0)\sum_{l\in\Z^d\setminus\{0\}}\frac{\overline{C}\overline{\epsilon}|\widehat{\mathcal{A}\delta\rho}(s,l)|+\overline{C}\epsilon_0|\widehat{\mathcal{A}\rho_2}(s,l)|}{|l|\langle l,sl\rangle^2}e^{-\delta\langle l\rangle^{1/3}}\lesssim \overline{C}^3\epsilon_0^2\overline{\eps}\langle s\rangle^{-6d}.
\end{equation*}
Similarly, recalling the definition \eqref{ner10} and using the Cauchy inequality in $\xi$, 
\begin{equation*}
\begin{split}
\delta IV_2(s)&\lesssim\sum_{k\in \mathbb{Z}^d}\sum_{l\in\Z^d\setminus\{0\}}e^{-\delta\langle k-l\rangle^{1/3}}\delta G_{\ast}(s,k)\frac{|\widehat{\mathcal{A}\delta\rho}(s,l)|G_{\ast1}(s,k-l)+|\widehat{\mathcal{A}\rho_2}(s,l)|\delta G_{\ast}(s,k-l)}{|l|\langle l,sl\rangle^2}\\
&\lesssim \overline{C}^3\epsilon_0^2\overline{\eps}\langle s\rangle^{-3}.
\end{split}
\end{equation*}
The last inequality holds since $\|G_{\ast 1}\|_{L^\infty_sL^2_k}\lesssim\overline{C}\overline{\epsilon}$, $\|\delta G_{\ast}\|_{L^\infty_sL^2_k}\lesssim\overline{C}\epsilon_0$, and using also \eqref{pam21}. The desired bounds follow from \eqref{ner21}, since $\overline{C}\leq\overline{\epsilon}^{-1/8}$.
\end{proof}

\begin{lemma}\label{Lem3ner}
For any $t\in[T_1,T_2]$, $L\geq 4$, and $i\in\{1,2\}$ we have
\begin{equation}\label{ner30}
\int_{T_1}^t\big[|III_i(s)|-|I_i(s)|/2\big]\,ds\leq\overline{\eps}^2\qquad\text{ and }\qquad \int_{T_1}^t\big[|\delta III(s)|-|\delta I(s)|/2\big]\,ds\leq\eps_0^2.
\end{equation}
\end{lemma}

\begin{proof} As before, the implied constants in this proof are allowed to depend on $d$ and $\lambda_0$. We provide all the details only for the slightly harder bounds on the difference of solutions.

We notice that there is derivative loss in the term $\delta III$, coming from the factor $\xi-tk$, which we can fortunately eliminate by symmetrization. For this, we notice that
\begin{equation*}
\begin{split}
\Re\Big\{&\sum_{k\in \mathbb{Z}^d}\sum_{l\in\Z^d\setminus\{0\}}\int_{\R^d}\langle k,\xi\rangle^{-2}A_L(t,k,\xi)\cdot \langle k-l,\xi-tl\rangle^{-2}A_L(t,k-l,\xi-tl)\\
&\qquad\qquad\times \Big[\,\overline{D^a_\xi\widehat{\delta g}(t,k,\xi)}\widehat{\rho_1}(t,l)D^a_\xi\widehat{\delta g}(t,k-l,\xi-tl)\frac{l\cdot(\xi-tk)}{|l|^2}\Big]\,d\xi\Big\}=0,
\end{split}
\end{equation*}
for any $t\in[T_1,T_2]$ and any multi-index $a$ satisfying $|a|\leq d'$. This follows easily by expanding $\Re z=(z+\overline{z})/2$ and observing that $\overline{\widehat{\rho_1}(t,l)}=\widehat{\rho_1}(t,-l)$ since $\rho_1$ is real-valued. Therefore
\begin{equation}\label{ner30.1}
|\delta III(s)|\lesssim \delta III_1(s)+\delta III_2(s)
\end{equation}
where
\begin{equation}\label{ner30.2}
\begin{split}
\delta III_1(s):=\sum_{|a|\leq d'}&\sum_{k\in \mathbb{Z}^d}\sum_{l\in\Z^d\setminus\{0\}}\int_{\R^d}\langle k,\xi\rangle^{-2}A_L(s,k,\xi)\Big|\frac{A_L(s,k,\xi)}{\langle k,\xi\rangle^2}-\frac{A_L(s,k-l,\xi-sl)}{\langle k-l,\xi-sl\rangle^2}\Big|\\
&\times\big|D^a_\xi\widehat{\delta g}(s,k,\xi)\big|\,|\widehat{\rho_1}(s,l)|\,\big|D^a_\xi\widehat{\delta g}(s,k-l,\xi-sl)\big|\frac{|\xi-sk|}{|l|}\,d\xi,
\end{split}
\end{equation}
\begin{equation}\label{ner30.3}
\begin{split}
\delta III_2(s):=\sum_{|a|\leq d'}&\sum_{k\in \mathbb{Z}^d}\sum_{l\in\Z^d\setminus\{0\}}\int_{\R^d}\langle k,\xi\rangle^{-4}A_L(s,k,\xi)^2\\
&\times\big|D^a_\xi\widehat{\delta g}(s,k,\xi)\big||\widehat{\delta\rho}(s,l)|\,\big|D^a_\xi\widehat{g_2}(s,k-l,\xi-sl)\big|\frac{|\xi-sk|}{|l|}\,d\xi.
\end{split}
\end{equation}

{\bf{Step 1.}} We bound first the contribution of the term $\delta III_1$. Recalling the definitions \eqref{pam3} and letting $\widehat{\mathcal{A}_{a,L}\delta g}(s,k,\xi):=A_L(s,k,\xi)D^a_\xi\widehat{\delta g}(s,k,\xi)$, we have
\begin{equation}\label{ner31}
\begin{split}
\delta III_1(s)\lesssim &\sum_{|a|\leq d'}\sum_{k\in \mathbb{Z}^d}\sum_{l\in\Z^d\setminus\{0\}}\int_{\R^d}\frac{|\widehat{\mathcal{A}_{a,L}\delta g}(s,k,\xi)|}{\langle k,\xi\rangle^2}\frac{|\widehat{\mathcal{A}_{a,L}\delta g}(s,k-l,\xi-sl)|}{\langle k-l,\xi-sl\rangle^2}\frac{|\widehat{\mathcal{A}\rho_1}(s,l)|}{|l|}\\
&\qquad\times\Big|\frac{A_L(s,k,\xi)}{\langle k,\xi\rangle^2}-\frac{A_L(s,k-l,\xi-sl)}{\langle k-l,\xi-sl\rangle^2}\Big|\frac{\langle k-l,\xi-sl\rangle^2|\xi-sk|}{A(s,l,sl)A_L(s,k-l,\xi-sl)}\,d\xi.
\end{split}
\end{equation}

We bound now the expression in the second line of \eqref{ner31}. If $10|l,sl|\geq|k-l,\xi-sl|$ then we use \eqref{gu6} to estimate
\begin{equation*}
\begin{split}
\Big|\frac{A_L(s,k,\xi)}{\langle k,\xi\rangle^2}-\frac{A_L(s,k-l,\xi-sl)}{\langle k-l,\xi-sl\rangle^2}\Big|\frac{\langle k-l,\xi-sl\rangle^2|\xi-sk|}{A(s,l,sl)A_L(s,k-l,\xi-sl)}&\lesssim |\xi-sk|e^{-2\delta |k-l,\xi-sl|^{1/3}}\\
&\lesssim \langle s\rangle e^{-\delta |k-l,\xi-sl|^{1/3}}.
\end{split}
\end{equation*}
On the other hand, if $|l,sl|\leq|k-l,\xi-sl|/10$ and $|l|\geq 1$ then we use \eqref{gu6.5} to estimate
\begin{equation*}
\begin{split}
\Big|\frac{A_L(s,k,\xi)}{\langle k,\xi\rangle^2}-\frac{A_L(s,k-l,\xi-sl)}{\langle k-l,\xi-sl\rangle^2}\Big|\frac{\langle k-l,\xi-sl\rangle^2|\xi-sk|}{A(s,l,sl)A_L(s,k-l,\xi-sl)}&\lesssim |\xi-sk|e^{-2\delta\langle l,sl\rangle^{1/3}}\langle k,\xi\rangle^{-2/3}\\
&\lesssim e^{-\delta\langle l,sl\rangle^{1/3}}\langle k,\xi\rangle^{1/3}.
\end{split}
\end{equation*}
Therefore
\begin{equation}\label{ner32}
\begin{split}
\delta III_1(s)&\lesssim J_{11}(s)+J_{12}(s),\\
J_{11}(s):=\sum_{|a|\leq d'}&\sum_{k\in \mathbb{Z}^d}\sum_{l\in\Z^d\setminus\{0\}}\int_{\R^d}\frac{|\widehat{\mathcal{A}_{a,L}\delta g}(s,k,\xi)|}{\langle k,\xi\rangle^2}\frac{|\widehat{\mathcal{A}_{a,L}\delta g}(s,k-l,\xi-sl)|}{\langle k-l,\xi-sl\rangle^2}\frac{|\widehat{\mathcal{A}\rho_1}(s,l)|}{|l|}\\
&\qquad\qquad\qquad\qquad\times \langle s\rangle e^{-\delta |k-l|^{1/3}}\,d\xi,\\
J_{12}(s):=\sum_{|a|\leq d'}&\sum_{k\in \mathbb{Z}^d}\sum_{l\in\Z^d\setminus\{0\}}\int_{\R^d}\frac{|\widehat{\mathcal{A}_{a,L}\delta g}(s,k,\xi)|}{\langle k,\xi\rangle^2}\frac{|\widehat{\mathcal{A}_{a,L}\delta g}(s,k-l,\xi-sl)|}{\langle k-l,\xi-sl\rangle^2}\frac{|\widehat{\mathcal{A}\rho_1}(s,l)|}{|l|}\\
&\qquad\qquad\qquad\qquad\times e^{-\delta\langle l,sl\rangle^{1/3}}\langle k,\xi\rangle^{1/6}\langle k-l,\xi-sl\rangle^{1/6}\,d\xi.
\end{split}
\end{equation}

We use the Cauchy inequality (in $\xi$), the definition \eqref{ner10}, and \eqref{pam21} to estimate
\begin{equation}\label{ner33}
J_{11}(s)\lesssim \sum_{k\in \mathbb{Z}^d}\sum_{l\in\Z^d\setminus\{0\}}(\delta G_{\ast})(s,k)(\delta G_{\ast})(s,k-l)\frac{|\widehat{\mathcal{A}\rho_1}(s,l)|\langle s\rangle}{|l|}e^{-\delta |k-l|^{1/3}}\lesssim\overline{C}^3\epsilon_0^2\overline{\eps}\langle s\rangle^{-2}.
\end{equation}
Moreover, the definition \eqref{ner2.3} and the identity \eqref{gu3.3} show that
\begin{equation*}
|\delta I(s)|\gtrsim (1+s)^{-\delta-1}\sum_{|a|\leq d'}\sum_{k\in \mathbb{Z}^d}\int_{\R^d}\langle k,\xi\rangle^{-4}A_L(s,k,\xi)^2\langle k,\xi\rangle^{1/3}\big|D^a_\xi\widehat{\delta g}(s,k,\xi)\big|^2\,d\xi.
\end{equation*}
Therefore $J_{12}(s)\lesssim (\overline{C}\overline{\epsilon})|\delta I(s)|$, using \eqref{pam20} and the Cauchy inequality in $(k,\xi)$. Finally, using also the bounds \eqref{ner32}--\eqref{ner33} we see that
\begin{equation}\label{ner35}
\int_0^t\big[|\delta III_1(s)|-|\delta I(s)|/2\big]\,ds\leq\eps_0^2/2.
\end{equation}

{\bf{Step 2.}} We bound now the contribution of $\delta III_2$. Recalling \eqref{pam3} and \eqref{ner10}, we have
\begin{equation}\label{ner36}
\begin{split}
\delta III_2(s)\lesssim &\sum_{|a|\leq d'}\sum_{k\in \mathbb{Z}^d}\sum_{l\in\Z^d\setminus\{0\}}\int_{\R^d}\frac{|\widehat{\mathcal{A}_a\delta g}(s,k,\xi)|}{\langle k,\xi\rangle^2}|\widehat{\mathcal{A}_a g_{2}}(s,k-l,\xi-sl)|\frac{|\widehat{\mathcal{A}\delta \rho}(s,l)|}{|l|\langle l,sl\rangle^2}\\
&\qquad\qquad\times\frac{\langle l,sl\rangle^2|\xi-sk|}{\langle k,\xi\rangle^2}\frac{A(s,k,\xi)}{A(s,l,sl)A(s,k-l,\xi-sl)}\,d\xi.
\end{split}
\end{equation}
It follows easily from \eqref{gu6} that
\begin{equation*}
\frac{\langle l,sl\rangle^2|\xi-sk|}{\langle k,\xi\rangle^2}\frac{A(s,k,\xi)}{A(s,l,sl)A(s,k-l,\xi-sl)}\lesssim \langle s\rangle\big[e^{-8\delta\langle l,sl\rangle^{1/3}}+e^{-8\delta\langle k-l,\xi-sl\rangle^{1/3}}\big].
\end{equation*}
As before, we can use the bootstrap assumptions \eqref{boot2}, the bounds \eqref{pam21}, and the Cauchy inequality to see that $|\delta III_2(s)|\lesssim \overline{C}^3\overline{\eps}\eps_0^2\langle s\rangle^{-2}$ for any $s\in[T_1,T_2]$. The desired bounds on the difference of solutions in \eqref{ner30} follow using also \eqref{ner35}.
\end{proof}

\section{Proof of Proposition \ref{MainBootstrapInfin}} \label{pamA1}

The proof of Proposition \ref{MainBootstrapInfin} is similar to the proof of Proposition \ref{MainBootstrap} in the previous section. As before, the constant $\overline{B}$ depends only on $d,\lambda_0,\vartheta$, and is thought of as much larger than other structural constants like $C_0^\ast$ in Lemma \ref{GreenF}, $C_0$ in Lemma \ref{g*}, and $B_1, B_2$ in Lemma \ref{LemInfin} and Proposition \ref{BootstrapInfin1} below. 

The conclusions follow from Propositions \ref{BootstrapInfin1} and \ref{BootstrapInfin4} below. We prove first bounds on the nonlinearities $\mathcal{N}'_i$ defined in \eqref{infin0}.

\begin{lemma}\label{LemInfin}
Assume that $g_1,g_2\in X^\sharp_{[T_1,T_2]}$ are as in Proposition \ref{MainBootstrapInfin}. Then
\begin{equation}\label{infin10}
\begin{split}
\sup_{t\in[T_1,T_2],\,k\in\Z^d\setminus\{0\}}A^\sharp(t,k,tk)|\widehat{\mathcal{N}'_i}(t,k)||k|^{-1/2}\langle t\rangle^{6d}&\leq B_1\overline{\theta},\\
\sup_{t\in[T_1,T_2],\,k\in\Z^d\setminus\{0\}}\langle k,tk\rangle^{-2}A^\sharp(t,k,tk)|(\widehat{\mathcal{N}'_1}-\widehat{\mathcal{N}'_2})(t,k)||k|^{-1/2}\langle t\rangle^{6d}&\leq B_1\theta_0,
\end{split}
\end{equation}
for $i\in\{1,2\}$ and a constant $B_1=B_1(d,\lambda_0)$ sufficiently large.
\end{lemma}

\begin{proof} This is similar to the proof of Lemma \ref{LemN}, since the weights $A^\sharp$ satisfy similar bounds as the weights $A$. The only non-trivial difference is that we need to replace the bounds \eqref{pam7} with different uniform bounds: for any $t\in[0,\infty)$ and $k\in\Z^d\setminus\{0\}$ we have
\begin{equation}\label{infin17}
\sum_{l\in\Z^d\setminus\{0\}}\int_t^{\infty}\frac{|t-s||k|^{1/2}\langle t\rangle^{6d}}{|l|^{1/2}\langle s\rangle^{6d}}\frac{A^\sharp(t,k,tk)}{A^\sharp(s,k,tk)}\big[e^{-8\delta\langle l,sl\rangle^{1/3}}+e^{-8\delta\langle k-l,tk-sl\rangle^{1/3}}\big]\,ds\lesssim 1.
\end{equation}
As before, we prove these bounds in three steps.

{\bf{Step 1.}} We show first that
\begin{equation}\label{infin20}
\begin{split}
\sum_{l\in\Z^d\setminus\{0\}}\int_t^\infty\frac{|t-s||k|^{1/2}\langle t\rangle^{6d}}{|l|^{1/2}\langle s\rangle^{6d}}\frac{A^\sharp(t,k,tk)}{A^\sharp(s,k,tk)}e^{-\delta(s^{1/3}+|l|^{1/3})}\,ds\lesssim 1.
\end{split}
\end{equation}
This is similar to the proof of \eqref{pam8}. The definition \eqref{infin3} shows that 
\begin{equation}\label{infin21}
\lambda^\sharp(s,r)-\lambda^\sharp(t,r)=\delta^2\int_t^s\big[(1+u)^{-\delta-1}\langle r\rangle^{1/3}+\big(1+u\langle r\rangle^{-2/3}\big)^{-\delta-1}\langle r\rangle^{-1/3}\big]\,du,
\end{equation}
if $0\leq t\leq s$ and $r\in[0,\infty)$. Therefore
\begin{equation}\label{infin22}
\frac{A^\sharp(t,k,tk)}{A^\sharp(s,k,tk)}=e^{\lambda^\ast(t,|k,tk|)-\lambda^\ast(s,|k.tk|)}\leq e^{-\delta^2|t-s|(1+s)^{-\delta-1}|k|^{1/3}\langle t\rangle^{1/3}}.
\end{equation}

Therefore, if $k,l\in\Z^d\setminus\{0\}$, and $s\geq t\geq 0$  then
\begin{equation*}
\begin{split}
\frac{|t-s||k|^{1/2}\langle t\rangle^{6d}}{|l|^{1/2}\langle s\rangle^{6d}}&\frac{A^\sharp(t,k,tk)}{A^\sharp(s,k,tk)}e^{-\delta(s^{1/3}+|l|^{1/3})}\\
&\lesssim e^{-\delta|l|^{1/3}}e^{-\delta s^{1/3}}|t-s||k|^{1/2}(1+|t-s||k|^{1/3}\langle s\rangle^{-1-\delta})^{-4}\\
&\lesssim e^{-\delta|l|^{1/3}}e^{-\delta s^{1/3}/2}|t-s||k|^{1/2}(1+|t-s||k|^{1/3})^{-4}.
\end{split}
\end{equation*}
The desired bounds \eqref{infin20} follow.

{\bf{Step 2.}} We show now that for any $t\geq 0$ and $k\in\Z^d\setminus\{0\}$ we have
\begin{equation}\label{infin24}
\sum_{l\in\Z^d\setminus\{0\}}\int_t^\infty\mathbf{1}_{\mathcal{R}'}(l,s)\frac{|t-s||k|^{1/2}\langle t\rangle^{6d}}{|l|^{1/2}\langle s\rangle^{6d}}\frac{A^\sharp(t,k,tk)}{A^\sharp(s,k,tk)}e^{-\delta(|k-l|^{1/3}+|tk-sl|^{1/3})}\,ds\lesssim 1,
\end{equation}
where
\begin{equation}\label{infin25}
\mathcal{R}':=\big\{(l,s)\in(\Z^d\setminus\{0\})\times[t,\infty):\,s\geq 1.01(t-1)\,\,\text{ or }\,\, |k-l|\geq |k|/10\big\}.
\end{equation}

Recall the identity \eqref{infin21}. Therefore, if $k,l\in\Z^d\setminus\{0\}$ and $s\geq \max\{t,1.01(t-1)\}$ then
\begin{equation}\label{infin26}
\frac{|t-s|\langle t\rangle^{6d}}{\langle s\rangle^{6d}}\frac{A^\sharp(t,k,tk)}{A^\sharp(s,k,tk)}e^{-\delta |tk-sl|^{1/3}}\lesssim \frac{|t-s|\langle t\rangle^{6d}}{\langle s\rangle^{6d}}e^{-\delta^4(1+t)^{1/4}}.
\end{equation}
Moreover, if $k,l\in\Z^d\setminus\{0\}$, $t\leq s\leq 1.01(t-1)$, and $|k-l|\geq |k|/10$ then $|tk-sl|=|s(k-l)+(t-s)k|\geq s|k-l|/2$, and the bounds \eqref{infin26} follow in this case as well. The desired bounds \eqref{infin24} follow from \eqref{pam12} and \eqref{infin26}.

{\bf{Step 3.}} Finally, we show that
\begin{equation}\label{infin28}
\sum_{l\in\Z^d\setminus\{0\}}\int_t^\infty(1-\mathbf{1}_{\mathcal{R}'}(l,s))\frac{|t-s||k|^{1/2}\langle t\rangle^{6d}}{|l|^{1/2}\langle s\rangle^{6d}}\frac{A^\sharp(t,k,tk)}{A^\sharp(s,k,tk)}e^{-\delta(|k-l|^{1/3}+|tk-sl|^{1/3})}\,ds\lesssim 1,
\end{equation}
where $\mathcal{R}'$ is as in \eqref{infin25}. We make the change of variables $l=k-a$, so it suffices to prove that if $k\in\Z^d\setminus\{0\}$ and $t\in[10,\infty]$ then
\begin{equation}\label{infin29}
\begin{split}
\sum_{a\in\Z^d,\,|a|\leq |k|/8}\int_t^{1.01t}\frac{|t-s|A^\sharp(t,k,tk)}{A^\sharp(s,k,tk)}e^{-\delta(|a|^{1/3}+|sa+(t-s)k|^{1/3})}\,ds\lesssim 1.
\end{split}
\end{equation}

We examine the identity \eqref{gu4} and use the second term in the right-hand side to see that
\begin{equation*}
\begin{split}
\lambda^\sharp(s,|k,tk|)-\lambda^\sharp(t,|k,tk|)&\geq \delta^2|t-s|(1+s\langle k,tk\rangle^{-2/3})^{-1-\delta}\langle k,tk\rangle^{-1/3}\\
&\geq\delta^3|t-s| t^{-1/3}|k|^{-1/3}(1+t^{1/3}|k|^{-2/3})^{-1-\delta},
\end{split}
\end{equation*}
if $t\geq 10$, $s\in[t,1.01t]$, and $k\in\Z^d\setminus\{0\}$. Therefore
\begin{equation}\label{infin30}
\frac{A^\sharp(t,k,tk)}{A^\sharp(s,k,tk)}\leq e^{-\delta^4|t-s|t^{-1/3}|k|^{-1/3}(1+t|k|^{-2})^{-1/3-\delta}}.
\end{equation}
Since $|sa+(t-s)k|\geq \big|s|a|-(s-t)|k|\big|$, for \eqref{pam16} it suffices to show that if $a\in\Z^d$, $k\in\Z^d\setminus\{0\}$, $t\in[10,T]$, and $|a|\leq|k|/8$ then
\begin{equation}\label{infin31}
\begin{split}
&\int_{t}^{1.01t}K'_a(t,k,s)\,ds\lesssim 1,\\
&K'_a(t,k,s):=|t-s|e^{-\delta^4|t-s|t^{-1/3}|k|^{-1/3}(1+t|k|^{-2})^{-1/3-\delta}}e^{-\delta|s|a|-(s-t)|k||^{1/3}}e^{-\delta |a|^{1/3}/2}.
\end{split}
\end{equation}

The bounds \eqref{infin31} are easy in the case $a=0$, since $|K'_0(t,k,s)|\lesssim |t-s|e^{-\delta|t-s|^{1/3}}\lesssim \langle t-s\rangle^{-4}$. On the other hand, if $a\in\Z^d\setminus\{0\}$ then
\begin{equation*}
|K'_a(t,k,s)|\lesssim |t-s|e^{-\delta|t-s|^{1/3}|k|^{1/3}/4}\lesssim \langle t-s\rangle^{-4}
\end{equation*}
if $(s-t)|k|\notin[t|a|/2,2t|a|]$, and
\begin{equation*}
\begin{split}
|K_a(t,k,s)|&\lesssim t|k|^{-1}e^{-\delta^4(t|k|^{-2})^{2/3}(1+t|k|^{-2})^{-1/3-\delta}/4}e^{-\delta|s|a|-(s-t)|k||^{1/3}}e^{-\delta |a|^{1/3}/4}\\
&\lesssim |k|e^{-\delta|s|a|-(s-t)|k||^{1/3}}e^{-\delta |a|^{1/3}/4}
\end{split}
\end{equation*}
if $(t-s)|k|\in[t|a|/2,2t|a|]$. The desired bounds \eqref{infin31} follow in this case as well.
\end{proof}

We can now combine these bounds with the identity \eqref{infin2} and the bounds \eqref{pam27.2} to prove our main estimates on the functions $\rho_1$, $\rho_2$, and $\rho_1-\rho_2$. The argument is identical to the argument in the proof of Proposition \ref{MainBootstrapIm1}.

\begin{proposition}\label{BootstrapInfin1} With the assumptions and notations of Proposition \ref{MainBootstrapInfin}, there is a constant $B_2=B_2(d,\lambda_0,\vartheta)\geq 1$ such that for any $t\in[T_1,T_2]$ and $i\in\{1,2\}$
\begin{equation}\label{infin35}
\mathcal{Z}_{g_i}^{\sharp,0}(t)\langle t\rangle^{6d}\leq B_2\overline{\theta}\qquad\text{ and }\qquad \mathcal{Z}_{g_1-g_2}^{\sharp,2}(t)\langle t\rangle^{6d}\leq B_2\theta_0.
\end{equation}
As a consequence, if $\overline{B}\geq B_2$ then there is a constant $C'_0=C'_0(d)$ such that
\begin{equation}\label{infin36}
\begin{split}
\big\|A^\sharp(t,k,tk)\widehat{\rho_i}(t,k)|k|^{-1/2}\|_{L^2_k}&\leq C'_0(\overline{B}^dB_2)^{1/(d+1)}\overline{\theta}\langle t\rangle^{-3},\\
\big\|\langle k,tk\rangle^{-2} A^\sharp(t,k,tk)(\widehat{\rho_1-\rho_2})(t,k)|k|^{-1/2}\|_{L^2_k}&\leq C'_0(\overline{B}^dB_2)^{1/(d+1)}\theta_0\langle t\rangle^{-3}
\end{split}
\end{equation}
for any $t\in[T_1,T_2]$ and $i\in\{1,2\}$.
\end{proposition}

As before, we note that the $L^2$ bounds \eqref{infin36} are needed later, in the proof of the estimates \eqref{infin55} in Proposition \ref{BootstrapInfin4}.

Finally we prove our main improved energy estimates.

\begin{proposition}\label{BootstrapInfin4}
With the assumptions and notations of Proposition \ref{MainBootstrapInfin}, for any $t\in[T_1,T_2]$ and $i\in\{1,2\}$ we have
\begin{equation}\label{infin39}
\mathcal{E}^{\sharp,0}_{g_i}(t)\leq\overline{B}^2\overline{\theta}^2/4\qquad\text{ and }\qquad\mathcal{E}^{\sharp,2}_{g_1-g_2}(t)\leq\overline{B}^2\theta_0^2/4.
\end{equation}
\end{proposition}

\begin{proof} This is similar to the proof of Proposition \ref{MainBootstrapIm2}. As before, to justify the calculations we define the mollified energies
\begin{equation*}
\begin{split}
&\mathcal{E}^{\sharp,p,L}_h(t):=\sum_{k\in \mathbb{Z}^d}\int_{\R^d}\langle k,\xi\rangle^{-2p}A^\sharp_L(t,k,\xi)^2\Big\{\sum_{|a|\leq d'}\big|D^a_\xi\widehat{h}(t,k,\xi)\big|^2\Big\}\,d\xi,\\
&A^\sharp_L(t,k,\xi):=(1+2^{-L}\langle k,\xi\rangle)^{-4}A^\sharp(t,k,\xi),
\end{split}
\end{equation*}
where $p\in[0,2]$, $L\geq 4$, and $h\in X^\sharp_{[T_1,T_2]}$. Then we calculate
\begin{equation}\label{infin40}
\frac{d}{dt}\mathcal{E}^{\sharp,0,L}_{g_i}(t)=I^\sharp_i(t)+II^\sharp_i(t)+III^\sharp_i(t)+IV^\sharp_i(t),
\end{equation}
where, with $\dot{A_L^\sharp}=\partial_tA_L^\sharp$,
\begin{equation}\label{infin41}
I^\sharp_i(t):=2\sum_{|a|\leq d'}\sum_{k\in \mathbb{Z}^d}\int_{\R^d}A_L^\sharp(t,k,\xi)\dot{A}_L^\sharp(t,k,\xi)\big|D^a_\xi\widehat{g_i}(t,k,\xi)\big|^2\,d\xi,
\end{equation}
\begin{equation}\label{infin42}
\begin{split}
II^\sharp_i(t)&:=-2\Re\Big\{\sum_{|a|\leq d'}\sum_{k\in \mathbb{Z}^d\setminus\{0\}}\int_{\R^d}A_L^\sharp(t,k,\xi)^2\overline{D^a_\xi\widehat{g_i}(t,k,\xi)}\\
&\qquad\qquad\times D^a_\xi\Big[\widehat{\rho_i}(t,k)\widehat{M_0}(\xi-tk)\frac{k\cdot (\xi-tk)}{|k|^2}\Big]\,d\xi\Big\},
\end{split}
\end{equation}
\begin{equation}\label{infin43}
\begin{split}
III^\sharp_i(t)&:=-\frac{2}{(2\pi)^d}\Re\Big\{\sum_{|a|\leq d'}\sum_{k\in \mathbb{Z}^d}\int_{\R^d}A_L^\sharp(t,k,\xi)^2\overline{D^a_\xi\widehat{g_i}(t,k,\xi)}\\
&\qquad\qquad\qquad\times \Big[\sum_{l\in\Z^d\setminus\{0\}}\widehat{\rho_i}(t,l)D^a_\xi\widehat{g_i}(t,k-l,\xi-tl)\frac{l\cdot(\xi-tk)}{|l|^2}\Big]\,d\xi\Big\},
\end{split}
\end{equation}
\begin{equation}\label{infin44}
\begin{split}
IV^\sharp_i(t)&:=-\frac{2}{(2\pi)^d}\Re\Big\{\sum_{|a|\leq d'}\sum_{|a'|=1,\,a'\leq a}\sum_{k\in \mathbb{Z}^d}\int_{\R^d}A_L^\sharp(t,k,\xi)^2\overline{D^a_\xi\widehat{g_i}(t,k,\xi)}\\
&\qquad\qquad\qquad\times \Big[\sum_{l\in\Z^d\setminus\{0\}}\widehat{\rho_i}(t,l)(D^{a-a'}_\xi\widehat{g_i})(t,k-l,\xi-tl)\frac{D^{a'}_\xi[l\cdot(\xi-tk)]}{|l|^2}\Big]\,d\xi\Big\}.
\end{split}
\end{equation}
Notice that $I^\sharp_i(t)\geq 0$, since $\dot{A^\sharp}(t,k,\xi)\geq 0$. Using \eqref{infin40} we can therefore estimate
\begin{equation}\label{infin45}
\begin{split}
&\mathcal{E}^{\sharp,0,L}_{g_i}(t)=\mathcal{E}^{\sharp,0,L}_{g_i}(T_2)-\int_t^{T_2}\big[I^\sharp_i(s)+II^\sharp_i(s)+III^\sharp_i(s)+IV^\sharp_i(s)\big]\,ds\\
&\leq\mathcal{E}^{\sharp,0}_{g_i}(T_2)+\int_t^{T_2}|II^\sharp_i(s)|\,ds+\int_t^{T_2}|IV^\sharp_i(s)|\,ds+\int_t^{T_2}\big[|III^\sharp_i(s)|-|I^\sharp_i(s)|\big]\,ds.
\end{split}
\end{equation}

Similarly, we calculate
\begin{equation}\label{infin46}
\frac{d}{dt}\mathcal{E}^{\sharp,2,L}_{g_1-g_2}(t)=\delta I^\sharp(t)+\delta II^\sharp(t)+\delta III^\sharp(t)+\delta IV^\sharp(t),
\end{equation}
where
\begin{equation}\label{infin47}
\delta I^\sharp(t):=2\sum_{|a|\leq d'}\sum_{k\in \mathbb{Z}^d}\int_{\R^d}\langle k,\xi\rangle^{-4}A_L^\sharp(t,k,\xi)\dot{A_L^\sharp}(t,k,\xi)\big|D^a_\xi(\widehat{g_1}-\widehat{g_2})(t,k,\xi)\big|^2\,d\xi,
\end{equation}
\begin{equation}\label{infin48}
\begin{split}
\delta II^\sharp(t)&:=-2\Re\Big\{\sum_{|a|\leq d'}\sum_{k\in \mathbb{Z}^d\setminus\{0\}}\int_{\R^d}\langle k,\xi\rangle^{-4}A_L^\sharp(t,k,\xi)^2\overline{D^a_\xi(\widehat{g_1}-\widehat{g_2})(t,k,\xi)}\\
&\qquad\qquad\times D^a_\xi\Big[(\widehat{\rho_1}-\widehat{\rho_2})(t,k)\widehat{M_0}(\xi-tk)\frac{k\cdot (\xi-tk)}{|k|^2}\Big]\,d\xi\Big\},
\end{split}
\end{equation}
\begin{equation}\label{infin49}
\begin{split}
&\delta III^\sharp(t):=-\frac{2}{(2\pi)^d}\Re\Big\{\sum_{|a|\leq d'}\sum_{k\in \mathbb{Z}^d}\int_{\R^d}\langle k,\xi\rangle^{-4}A_L^\sharp(t,k,\xi)^2\overline{D^a_\xi(\widehat{g_1}-\widehat{g_2})(t,k,\xi)}\\
&\quad\times \Big[\sum_{l\in\Z^d\setminus\{0\}}[\widehat{\rho_1}(t,l)D^a_\xi\widehat{g_1}(t,k-l,\xi-tl)-\widehat{\rho_2}(t,l)D^a_\xi\widehat{g_2}(t,k-l,\xi-tl)]\frac{l\cdot(\xi-tk)}{|l|^2}\Big]\,d\xi\Big\},
\end{split}
\end{equation}
\begin{equation}\label{infin50}
\begin{split}
&\delta IV^\sharp(t):=-\frac{2}{(2\pi)^d}\Re\Big\{\sum_{|a|\leq d'}\sum_{|a'|=1,\,a'\leq a}\sum_{k\in \mathbb{Z}^d}\int_{\R^d}\langle k,\xi\rangle^{-4}A_L^\sharp(t,k,\xi)^2\overline{D^a_\xi(\widehat{g_1}-\widehat{g_2})(t,k,\xi)}\\
&\quad\times \Big[\sum_{l\in\Z^d\setminus\{0\}}[\widehat{\rho_1}(t,l)(D^{a-a'}_\xi\widehat{g_1})(t,k-l,\xi-tl)-\widehat{\rho_2}(t,l)(D^{a-a'}_\xi\widehat{g_2})(t,k-l,\xi-tl)]\\
&\qquad\qquad\qquad\times\frac{D^{a'}_\xi[l\cdot(\xi-tk)]}{|l|^2}\Big]\,d\xi\Big\}.
\end{split}
\end{equation}
As before, we notice that $\delta I^\sharp(t)\geq 0$, so
\begin{equation}\label{infin51}
\begin{split}
\mathcal{E}^{\sharp,2,L}_{g_1-g_2}(t)&\leq\mathcal{E}^{\sharp,2}_{g_1-g_2}(T_2)+\int_t^{T_2}|\delta II^\sharp(s)|\,ds+\int_t^{T_2}|\delta IV^\sharp(s)|\,ds\\
&+\int_t^{T_2}\big[|\delta III^\sharp(s)|-|\delta I^\sharp(s)|\big]\,ds.
\end{split}
\end{equation}

As in the proof of Lemma \ref{Lem1ner}, we can use the bounds \eqref{ner9.5} and \eqref{infin36} to show that 
\begin{equation}\label{infin55}
\int_t^{T_2}\big|II^\sharp_i(s)\big|\,ds\leq\overline{B}^2\overline{\theta}^2/8\qquad\text{ and }\qquad \int_t^{T_2}\big|\delta II^\sharp(s)\big|\,ds\leq\overline{B}^2\theta_0^2/8,
\end{equation}
for any $t\in[T_1,T_2]$ and $i\in\{1,2\}$. Moreover, as in the proofs of Lemmas \ref{Lem2ner} and \ref{Lem3ner}, we can use the bounds \eqref{gu6} and \eqref{gu6.5} (which are the same for the weights $A$ and $A^\sharp$) to show that
\begin{equation*}
\int_t^{T_2}\big|IV^\sharp_i(s)\big|\,ds\leq\overline{\theta}^2\qquad\text{ and }\qquad \int_t^{T_2}\big|\delta IV^\sharp(s)\big|\,ds\leq\theta_0^2,
\end{equation*}
\begin{equation*}
\int_t^{T_2}\big[|III_i^\sharp(s)|-|I_i^\sharp(s)|\big]\,ds\leq\overline{\theta}^2\qquad\text{ and }\qquad \int_t^{T_2}\big[|\delta III^\sharp(s)|-|\delta I^\sharp(s)|\big]\,ds\leq\theta_0^2.
\end{equation*}
The desired bounds \eqref{infin39} follow from these last three estimates.
\end{proof}

\section{Proofs of the main theorems}\label{proofTHM}

We can now prove our three main theorems.

\subsection{Proof of Theorem \ref{MainThm}}  We use the following simple continuation criterion:

\begin{lemma}\label{LemmaProo}
Assume $T_1\leq T_2\in [0,\infty)$, $\lambda_1\in[0.6\lambda_0,0.9\lambda_0]$, and $g\in X_{[T_1,T_2]}$ is a real-valued solution of the system \eqref{NVP3} satisfying $\widehat{g}(t,0,0)=0$ for any $t\in[T_1,T_2]$. Assume that
\begin{equation}\label{proo3}
\sup_{t\in[T_1,T_2]}\mathcal{E}^0_g(t)\leq\rho_0^2,
\end{equation}
where $\rho_0=\rho_0(d,\lambda_0,\vartheta)$ is sufficiently small. Then $g$ can be extended to a real-valued solution $g'\in X_{[T_1,T'_2]}$ of the system \eqref{NVP3}, for $T'_2=T_2+c_{d,\lambda_0,\vartheta}(1+T_2)^{-2}$, satisfying
\begin{equation}\label{proo5}
\sup_{t\in[T_1,T'_2]}\mathcal{E}^0_g(t)\leq 2\rho_0^2.
\end{equation}
\end{lemma}

\begin{proof} This is a local well-posedness result and we can use a standard parabolic regularization argument to prove it. We provide some details for the sake of completeness.

For $\varepsilon>0$ we are looking for solutions $g^\varepsilon\in X_{[T_2,T_2^\varepsilon]}$ of the equation (in the Fourier space)
\begin{equation}\label{proo2}
\begin{split}
&\partial_t\widehat{g^\varepsilon}(t,k,\xi)+\varepsilon(|k|^2+|\xi|^2)\widehat{g^\varepsilon}(t,k,\xi)+\widehat{g^\varepsilon}(t,k,tk)\widehat{M_0}(\xi-tk)\frac{k\cdot (\xi-tk)}{|k|^2}\\
&\qquad+\frac{1}{(2\pi)^d}\sum_{l\in\Z^d\setminus\{0\}}\widehat{g^\varepsilon}(t,l,tl)\widehat{g^\varepsilon}(t,k-l,\xi-tl)\frac{l\cdot(\xi-tk)}{|l|^2}=0,
\end{split}
\end{equation}
with initial data $\widehat{g^\varepsilon}(T_2)=\widehat{g}(T_2)$ (compare with \eqref{NVP4}). Such solutions can be constructed by a fixed-point argument in the space $X_{[T_2,T_2^\varepsilon]}$, using the identity
\begin{equation}\label{proo9}
\widehat{g^\varepsilon}(t,k,\xi)=e^{-\varepsilon(|k|^2+|\xi|^2)(t-T_2)}\widehat{g}(T_2,k,\xi)-\sum_{i\in\{1,2\}}\int_{T_2}^te^{-\varepsilon(|k|^2+|\xi|^2)(t-s)}\mathcal{M}_i(\widehat{g^\varepsilon})(s,k,\xi)\,ds,
\end{equation}
where $T_2^\varepsilon-T_2$ is sufficiently small (depending on $\varepsilon, T_2, d, \lambda_0, \vartheta$) and
\begin{equation}\label{proo8}
\begin{split}
&\mathcal{M}_1(f)(t,k,\xi):=f(t,k,tk)\widehat{M_0}(\xi-tk)\frac{k\cdot (\xi-tk)}{|k|^2},\\
&\mathcal{M}_2(f)(t,k,\xi):=\frac{1}{(2\pi)^d}\sum_{l\in\Z^d\setminus\{0\}}f(t,l,tl)f(t,k-l,\xi-tl)\frac{l\cdot(\xi-tk)}{|l|^2}.
\end{split}
\end{equation}
We may assume that the solutions $g^\varepsilon$ are real-valued and $\mathcal{E}^0_{g^\varepsilon}(t)\leq 1.5\rho_0^2$ for any $t\in[T_2,T_2^\varepsilon]$. As in the proof of Proposition \ref{MainBootstrapIm2}, we can estimate the energy increment by
\begin{equation*}
\frac{d}{dt}\mathcal{E}_{g^\varepsilon}^0(t)\lesssim_{d,\lambda_0,\vartheta} (1+T_2)^2\mathcal{E}_{g^\varepsilon}^0(t).
\end{equation*}
Therefore the solution $g^\varepsilon$ of \eqref{proo2} can be extended on the time interval $[T_2,T'_2]$, $T'_2:=T_2+c_{d,\lambda_0,\vartheta}(1+T_2)^{-2}$, satisfying the uniform bounds $\mathcal{E}^0_{g^\varepsilon}(t)\leq 1.5\rho_0^2$ for any $t\in[T_2,T'_2]$.

We can now let $\varepsilon\to 0$ to construct the desired extension $g$. Indeed, let
\begin{equation*}
\mathcal{L}_{\varepsilon,\varepsilon'}(t):=\sum_{|a|\leq d'}\sum_{k\in \mathbb{Z}^d}\int_{\R^d}\big|D^a_\xi\widehat{(g^\varepsilon-g^{\varepsilon'})}(t,k,\xi)\big|^2\,d\xi.
\end{equation*}
As in the proof of Proposition \ref{MainBootstrapIm2}, we can estimate
\begin{equation*}
|\partial_t\mathcal{L}_{\varepsilon,\varepsilon'}(t)|\lesssim_{d,\lambda_0,\vartheta}(\varepsilon+\varepsilon')[\mathcal{L}_{\varepsilon,\varepsilon'}(t)]^{1/2}+\mathcal{L}_{\varepsilon,\varepsilon'}(t),
\end{equation*}
for any $t\in[T_2,T'_2]$. Since $\mathcal{L}_{\varepsilon,\varepsilon'}(T_2)=0$ it follows that $\lim_{\varepsilon,\varepsilon'\to 0}\|\mathcal{L}_{\varepsilon,\varepsilon'}(t)\|_{L^\infty_t}=0$. Therefore the sequence $g^\varepsilon$ converges to a solution $g\in X_{[T_1,T'_2]}$ of the equation \eqref{NVP4}, which satisfies the uniform bounds \eqref{proo5}, as desired.
\end{proof}

\begin{proof}[Proof of Theorem \ref{MainThm}] First, we combine Lemma \ref{LemmaProo} and the bootstrap Proposition \ref{MainBootstrap} (both with $\lambda_1=0.9\lambda_0$) to construct a unique solution $g\in C([0,\infty):\mathcal{G}_{d'}^{\lambda_0/2,1/3}(\T^d\times\R^d))$ of the equation \eqref{NVP3} with initial data $g(0)=f_0$, satisfying 
\begin{equation}\label{proo30}
\sum_{|a|\leq d'}\|D^a_\xi\widehat{g}(t,k,\xi)\cdot A(t,k,\xi)\|_{L^2_{k,\xi}}\lesssim\kappa_0,\qquad \int_{\T^d\times\R^d} g(t,x,v)\,dxdv=0,
\end{equation} 
for any $t\in[0,\infty)$, where the implied constants in this proof are allowed to depend only on $d,\lambda_0,\vartheta$. In particular, using Lemma \ref{g*}, for any $t\in[0,\infty)$ we have
\begin{equation}\label{proo31}
\|\widehat{\rho}(t,k)\cdot A(t,k,tk)\|_{L^2_k}\lesssim \kappa_0.
\end{equation}
Notice that $A(t,k,tk)\geq e^{0.9\lambda_0\langle k,tk\rangle^{1/3}}\geq e^{0.6\lambda_0\langle k\rangle^{1/3}}e^{0.3\lambda_0\langle t\rangle^{1/3}}$ if $k\in\Z^d\setminus \{0\}$. Thus, using \eqref{proo31},
\begin{equation}\label{proo32}
\|\rho(t)\|_{\mathcal{G}^{0.6\lambda_0,1/3}(\T^d)}\lesssim \kappa_0e^{-0.3\lambda_0\langle t\rangle^{1/3}}.
\end{equation}

Moreover, using the identity \eqref{NVP4}, for any $t\leq t'\in[0,\infty)$ we have
\begin{equation}\label{proo33}
\begin{split}
\widehat{g}(t',k,\xi)-\widehat{g}(t,k,\xi)&+\int_{t}^{t'}\widehat{\rho}(s,k)\widehat{M_0}(\xi-sk)\frac{k\cdot (\xi-sk)}{|k|^2}\,ds\\
&+\frac{1}{(2\pi)^d}\sum_{l\in\Z^d}\int_{t}^{t'}\widehat{\rho}(s,l)\widehat{g}(s,k-l,\xi-sl)\frac{l\cdot(\xi-sk)}{|l|^2}\,ds=0.
\end{split}
\end{equation}
Therefore, recalling the definitions \eqref{pam3}, \eqref{ner8}, and \eqref{ner10}, 
\begin{equation}\label{proo33.5}
\begin{split}
\sum_{|a|\leq d}&e^{0.6\lambda_0\langle k,\xi\rangle^{1/3}}\big|D^a_\xi[\widehat{g}(t',k,\xi)-\widehat{g}(t,k,\xi)]\big|\lesssim \int_{t}^{t'}|\widehat{A\rho}(s,k)|M_0^\ast(\xi-sk)\frac{e^{0.6\lambda_0\langle k,\xi\rangle^{1/3}}}{A(s,k,sk)}\,ds\\
&+\sum_{|a|\leq d}\sum_{l\in\Z^d\setminus\{0\}}\int_{t}^{t'}|\widehat{A\rho}(s,l)||\widehat{Ag_{a}}(s,k-l,\xi-sl)|\frac{|\xi-sk|e^{0.6\lambda_0\langle k,\xi\rangle^{1/3}}}{A(s,l,sl)A(s,k-l,\xi-sl)}\,ds.
\end{split}
\end{equation}
Notice that
\begin{equation}\label{proo34}
\begin{split}
\frac{e^{0.6\lambda_0\langle k,\xi\rangle^{1/3}}}{A(s,k,sk)}&\lesssim \frac{e^{0.6\lambda_0\langle k,sk\rangle^{1/3}}e^{0.6\lambda_0\langle\xi-sk\rangle^{1/3}}}{A(s,k,sk)}\lesssim \frac{e^{0.6\lambda_0\langle\xi-sk\rangle^{1/3}}}{e^{0.3\lambda_0\langle k,sk\rangle^{1/3}}},\\
\frac{|\xi-sk|e^{0.6\lambda_0\langle k,\xi\rangle^{1/3}}}{A(s,l,sl)A(s,k-l,\xi-sl)}&\lesssim \frac{|\xi-sk|e^{0.6\lambda_0\langle l,sl\rangle^{1/3}}e^{0.6\lambda_0\langle k-l,\xi-sl\rangle^{1/3}}}{A(s,l,sl)A(s,k-l,\xi-sl)}\lesssim\frac{1}{e^{0.29\lambda_0\langle l,sl\rangle^{1/3}}}.
\end{split}
\end{equation}
Notice that $\|\widehat{Ag_a}(s,k,\xi)\|_{L^2_{k,\xi}}+\|\widehat{A\rho}(s,k)\|_{L^2_{k}}\lesssim \kappa_0$, due to \eqref{proo30}--\eqref{proo31}. Recalling also the bounds \eqref{ner9.5}, it follows from the last two estimates that
\begin{equation*}
\begin{split}
\sum_{|a|\leq d}\big\|e^{0.6\lambda_0\langle k,\xi\rangle^{1/3}}D^a_\xi[\widehat{g}(t',k,\xi)-\widehat{g}(t,k,\xi)]\big\|_{L^2_{k,\xi}}\lesssim \int_{t}^{t'}\kappa_0e^{-0.29\lambda_0\langle s\rangle^{1/3}}\,ds\lesssim \kappa_0e^{-0.28\lambda_0\langle t\rangle^{1/3}}.
\end{split}
\end{equation*}
Therefore, if $t\leq t'\in[0,\infty)$, 
\begin{equation}\label{proo35}
\|g(t')-g(t)\|_{\mathcal{G}_{d'}^{0.6\lambda_0,1/3}}\lesssim \kappa_0e^{-0.28\lambda_0\langle t\rangle^{1/3}}.
\end{equation}

In particular, we can define $g_\infty:=\lim_{t\to\infty}g(t)$ in $\mathcal{G}_{d'}^{0.6\lambda_0,1/3}$, and the desired bounds \eqref{MainTHM2} follow from \eqref{proo32} and \eqref{proo35}. In fact $\|g_\infty\|_{\mathcal{G}_{d'}^{0.8\lambda_0,1/3}}\lesssim\kappa_0$, using again the uniforms bounds \eqref{proo30}. This completes the proof of Theorem \ref{MainThm} (i). 

To prove part (ii), assume $f_{01},f_{02}\in B_{\overline{\kappa}}(\mathcal{G}_{d'}^{\lambda_0,1/3})$ are initial data and let $g_1,g_2\in C([0,\infty):\mathcal{G}_{d'}^{\lambda_0/2,1/3})$ denote the corresponding solutions of the system \eqref{NVP3}. We first set $\lambda_1=0.9\lambda_0$ and notice that the function $t\mapsto \mathcal{E}^2_{g_1-g_2}(t)$ is a continuous map from $[0,\infty)$ to $[0,1]$ (since $\mathcal{E}^0_{g_i}(t)\lesssim\overline{\kappa}^2$ for any $t\in[0,\infty)$ and $i\in\{1,2\}$). In view of Proposition \ref{MainBootstrap} it follows that 
\begin{equation*}
\mathcal{E}^2_{g_1-g_2}(t)\lesssim \|f_{01}-f_{02}\|^2_{\mathcal{G}_{d'}^{\lambda_0,1/3}}\qquad\text{ for any }t\in[0,\infty).
\end{equation*}
This gives the second inequality in \eqref{MainTHM4}. To prove the reverse inequality we set $\lambda_1=0.7\lambda_0$, notice that the function $t\mapsto \mathcal{E}^{\sharp,2}_{g_1-g_2}(t)$ is a continuous map from $[0,\infty)$ to $[0,1]$, and use Proposition \ref{MainBootstrapInfin} to prove that $\mathcal{E}^{\sharp,2}_{g_1-g_2}(t')\lesssim \|(g_1-g_2)(t)\|^2_{\mathcal{G}_{d'}^{3\lambda_0/4,1/3}}$ for any $t'\leq t\in[0,\infty)$. This completes the proof of Theorem \ref{MainThm} (ii).
\end{proof}

\subsection{Proof of Theorem \ref{MainThm22}}  We need again a continuation criterion, similar to Lemma \ref{LemmaProo}.

\begin{lemma}\label{LemmaProo2}
Assume $T_1\leq T_2\in [0,\infty)$, $\lambda_1\in[0.6\lambda_0,0.9\lambda_0]$, and $g\in X^\sharp_{[T_1,T_2]}$ is a real-valued solution of the system \eqref{NVP3} satisfying $\widehat{g}(t,0,0)=0$ for any $t\in[T_1,T_2]$. Assume that
\begin{equation}\label{proo51}
\sup_{t\in[T_1,T_2]}\mathcal{E}^{\sharp,0}_g(t)\leq\rho_0^2,
\end{equation}
where $\rho_0=\rho_0(d,\lambda_0,\vartheta)$ is sufficiently small. Then $g$ can be extended to a real-valued solution $g'\in X^\sharp_{[T'_1,T_2]}$ of the system \eqref{NVP3}, for $T'_1=\max\{0,T_1-c_{d,\lambda_0,\vartheta}(1+T_1)^{-2}\}$, satisfying
\begin{equation}\label{proo52}
\sup_{t\in[T'_1,T_1]}\mathcal{E}^{\sharp,0}_g(t)\leq 2\rho_0^2.
\end{equation}
\end{lemma}

The proof of this lemma is similar to the proof of Lemma \ref{LemmaProo}. 

To prove Theorem \ref{MainThm22} we start with a final state data $g_\infty$ satisfying \eqref{MainTHMXY1}, set $\lambda_1=0.9\lambda_0$,  and construct a sequence of function $g_n\in C\big([0,\infty):\mathcal{G}_{d'}^{\lambda_0/2,1/3}\big)$, $n\geq 1$, satisfying the following properties:
\smallskip

(1) For any $t\in[n,\infty)$ we have $g_n(t)=g_\infty$;

(2) The function $g_n\in C\big([0,n]:\mathcal{G}_{d'}^{\lambda_0/2,1/3}\big)$ is a solution of the system \eqref{NVP3}, satisfying
\begin{equation}\label{proo53}
\sum_{|a|\leq d'}\|D^a_\xi\widehat{g_n}(t,k,\xi)\cdot A^\sharp(t,k,\xi)\|_{L^2_{k,\xi}}\lesssim\overline{\kappa}'\quad\text{ and }\quad\widehat{g_n}(t,0,0)=0\quad\text{ for any }t\in[0,n].
\end{equation}

To prove property (2) we set $g_n(n)=g_\infty$ and use Proposition \ref{MainBootstrapInfin} and Lemma \ref{LemmaProo2} to construct the solution $g_n\in C\big([0,n]:\mathcal{G}_{d'}^{\lambda_0/2,1/3}\big)$ of the system \eqref{NVP3} satisfying the bounds \eqref{proo53}.

We show now that $\{g_n\}_{n\geq 1}$ is a Cauchy sequence in $C\big([0,\infty):\mathcal{G}_{d'}^{0.75\lambda_0,1/3}\big)$. Indeed assume that $n\leq n'\in\mathbb{Z}_+$ are sufficiently large. Clearly $g_n(t)=g_{n'}(t)=g_\infty$ if $t\geq n'$. Moreover, if $t\in[n,n']$ then we use the identity \eqref{proo33}, in the form
\begin{equation*}
\begin{split}
\widehat{g_{n'}}(n',k,\xi)-\widehat{g_{n'}}(t,k,\xi)&+\int_{t}^{n'}\widehat{\rho_{n'}}(s,k)\widehat{M_0}(\xi-sk)\frac{k\cdot (\xi-sk)}{|k|^2}\,ds\\
&+\frac{1}{(2\pi)^d}\sum_{l\in\Z^d}\int_{t}^{n'}\widehat{\rho_{n'}}(s,l)\widehat{g_{n'}}(s,k-l,\xi-sl)\frac{l\cdot(\xi-sk)}{|l|^2}\,ds=0.
\end{split}
\end{equation*}
Then we use the uniform bounds \eqref{proo53} and estimate as in \eqref{proo33.5}--\eqref{proo34} to show that
\begin{equation*}
\|g_{n'}(n')-g_{n'}(t)\|_{\mathcal{G}_{d'}^{0.85\lambda_0,1/3}}\lesssim \overline{\kappa}'e^{-\delta\langle t\rangle^{1/3}}.
\end{equation*}

Since $g_n(t)=g_\infty=g_{n'}(n')$, this shows that $\|g_n(t)-g_{n'}(t)\|_{\mathcal{G}_{d'}^{0.85\lambda_0,1/3}}\lesssim \overline{\kappa}'e^{-\delta\langle t\rangle^{1/3}}$ if $t\in[n,n']$. We can then use Proposition \ref{MainBootstrapInfin} on the time interval $[T_1,T_2]=[0,n]$, with $\lambda_1=0.8\lambda_0$, to show that $\|g_n(t)-g_{n'}(t)\|_{\mathcal{G}_{d'}^{0.75\lambda_0,1/3}}\lesssim \overline{\kappa}'e^{-\delta \langle n\rangle^{1/3}}$ if $t\in[0,n]$. Therefore
\begin{equation}\label{proo55}
\sup_{t\in[0,\infty)}\|g_n(t)-g_{n'}(t)\|_{\mathcal{G}_{d'}^{0.75\lambda_0,1/3}}\lesssim \overline{\kappa}'e^{-\delta\langle n\rangle^{1/3}}\qquad\text{ for any }n\leq n'\in\Z_+.
\end{equation}
In particular, recalling the uniform bounds \eqref{proo53}, $\{g_n\}_{n\geq 1}$ is a Cauchy sequence in $C\big([0,\infty):\mathcal{G}_{d'}^{3\lambda_0/4,1/3}\big)$ and its limit $g\in C\big([0,\infty):\mathcal{G}_{d'}^{3\lambda_0/4,1/3}\big)$ is a solution of the final state problem. 

The uniqueness of this solution follows from the quantitative bounds \eqref{MainTHM4}. The bounds \eqref{MainTHMXY4} follow also as in the proof of Theorem \ref{MainThm}, using Proposition \ref{MainBootstrap} for the first inequality and Proposition \eqref{MainBootstrapInfin} for the second inequality. This completes the proof of Theorem \ref{MainThm22}. 

\subsection{Proof of Theorem \ref{MainThm33}} We define $h_\infty(x,v):=g_\infty(-x,v)$ and use Theorem \ref{MainThm22} to construct the corresponding solution $h\in C([0,\infty):\mathcal{G}_{d'}^{3\lambda_0/4,1/3})$ of the system \eqref{NVP3}. Then we define $g(t,x,v):=h(-t,-x,v)$, and notice that $g\in C([-\infty, 0]:\mathcal{G}_{d'}^{3\lambda_0/4,1/3})$ is a solution of the system \eqref{NVP3} on the time interval $(-\infty,0]$. Finally, we extend the solution $g$ to the time interval $[0,\infty)$, using Theorem \ref{MainThm}, and define the corresponding final state $g_\infty$. The desired bounds \eqref{MainTH4} follow from \eqref{MainTHMXY4} and \eqref{MainTHM4}. 

\section{Proof of Lemma \ref{GreenF}}\label{ProofGreen}

In this section we provide a self-contained proof of Lemma \ref{GreenF}. The implied constants in this proof are allowed to depend only on the parameters $d$, $\lambda_0$, and $\vartheta$. 

(i) Using the definition \eqref{gu14} and integration by parts, we have
\begin{equation}\label{gu14.5}
i\tau L(\tau,k)=\int_0^\infty e^{-i\tau s}\big[\widehat{M_0}(sk)+sk\cdot\nabla\widehat{M_0}(sk)\big]\,ds.
\end{equation}

Using Plancherel theorem and \eqref{M01} it follows that
\begin{equation}\label{gu15.1}
\big\|\langle \alpha\rangle L(\alpha+i\beta,k)\big\|_{L^2_\alpha}\lesssim |k|^{-1/2}\qquad\text{ for any }k\in\Z^d\setminus\{0\},\,\beta\in(-\infty,0].
\end{equation}
The $L^2_\alpha$ estimates in the first line of \eqref{gu19.1} follow using the assumption \eqref{eq:Penrose}.

Moreover, for any integer $a\geq 0$ we use \eqref{M01}, \eqref{gu14}, and \eqref{gu14.5} to estimate
\begin{equation*}
\big|D^a_\tau L(\tau,k)\big|\lesssim \int_0^\infty s^{a+1}e^{-\lambda_0|sk|^{1/3}}\,ds\lesssim \frac{1}{(|k|\lambda_0^3)^{a+2}}\int_0^\infty t^{3a+5}e^{-t}\,ds\lesssim \frac{(3a+5)!}{(|k|\lambda_0^3)^{a+2}}.
\end{equation*}
\begin{equation*}
\big|D^a_\tau (\tau L(\tau,k))\big|\lesssim \int_0^\infty s^ae^{-\lambda_0|sk|^{1/3}}\,ds\lesssim  \frac{(3a+2)!}{(|k|\lambda_0^3)^{a+1}}.
\end{equation*}
Therefore, for any $k\in\Z^d\setminus\{0\}$, $\beta\in(-\infty,0]$, and $a\in\Z_+$ we have
\begin{equation}\label{gu15.2}
\big\|\langle\alpha\rangle D^a_\alpha L(\alpha+i\beta,k)\big\|_{L^\infty_\alpha}\lesssim \frac{(3a+5)!}{(|k|\lambda_0^3)^{a+1}}.
\end{equation}

It remains to prove pointwise estimates on higher order derivatives of $L'$, of the form
\begin{equation}\label{pam25}
\big\|\langle\alpha\rangle D^a_\alpha L'(\alpha+i\beta,k)\big\|_{L^\infty_\alpha}\lesssim \frac{(3a)!(1.01)^{a}}{(|k|\lambda_0^3)^{a+1}},
\end{equation}
for any $k\in\Z^d\setminus\{0\}$, $\beta\in(-\infty,0]$, and $a\in\Z_+$. We start from the Fa\'{a} di Bruno's formula 
\begin{equation}\label{gu20}
D^n(g\circ f)(x)=\sum_\ast\frac{n!}{m_1!(1!)^{m_1}\cdot\ldots\cdot m_n!(n!)^{m_n}}D^{m_1+\ldots+m_n}g(f(x))\cdot\prod_{j=1}^n(D^jf(x))^{m_j},
\end{equation}
for any $n\geq 1$ and any sufficiently smooth functions $f:\R\to\C$ and $g:\C\to\C$, where $\sum\limits_\ast$ denotes the sum over all $n$-tuples of non-negative integers $m_1,\ldots,m_n$ satisfying
\begin{equation}\label{gu21}
1\cdot m_1+\ldots+n\cdot m_n=n.
\end{equation}
In our case, we fix $\beta\in(-\infty,0]$, $k\in\Z^d\setminus\{0\}$, and $g(y)=y/(y+1)$. Clearly $|D^bg(y)|\leq b!|y+1|^{-b-1}$ for any integer $b\geq 0$, thus
\begin{equation*}
|D^n_\alpha L'(\alpha+i\beta,k)|\leq \sum_\ast\frac{n!(m_1+\ldots+m_n)!}{m_1!\cdot\ldots\cdot m_n!\vartheta^{m_1+\ldots+m_n+1}}\prod_{j=1}^n\Big[\frac{D^jL(\alpha+i\beta,k)}{j!}\Big]^{m_j},
\end{equation*}
for any integer $n\geq 1$. It follows from \eqref{gu15.2} that 
\begin{equation}\label{gu22}
\langle\alpha\rangle |D^n_\alpha L'(\alpha+i\beta,k)|\leq \sum_\ast [C_1^\ast]^{m_1+\ldots+m_n}\frac{n!(m_1+\ldots+m_n)!}{m_1!\cdot\ldots\cdot m_n!}\prod_{j=1}^n\Big[\frac{(3j+5)!}{j!(|k|\lambda_0^3)^{j+1}}\Big]^{m_j},
\end{equation}
for any $\alpha\in\R$ and $n\in\Z_+^\ast$, where $C_1^\ast=C_1^\ast(d,\lambda_0,\vartheta)$ is a constant. Thus
\begin{equation}\label{gu24}
\begin{split}
&\frac{\langle\alpha\rangle |D^n_\alpha L'(\alpha+i\beta,k)|(|k|\lambda_0^3)^{n+1}}{(3n)!}\leq \sum_\ast \frac{[C_2^\ast]^{m_1+\ldots+m_n}}{m_1!\cdot\ldots\cdot m_n!}\exp(H_{m_1,\ldots,m_n}),\\
&H_{m_1,\ldots,m_n}:=[\ln(n!)-\ln((3n)!)]+\sum_{j=1}^nm_j[\ln((3j+5)!)-\ln(j!)]+\big[\ln((m_1+\ldots+m_n)!)\big],
\end{split}
\end{equation}
where $C_2^\ast=C_2^\ast(d,\lambda_0,\vartheta)$ is a constant.

To estimate the sum in the right-hand side of \eqref{gu24} we notice that
\begin{equation*}
\Big|\ln 1+\ldots+\ln m-\int_{1/2}^{m+1/2}\ln x\,dx\Big|\leq C\qquad\text{ for any }m\geq 1,
\end{equation*}
where $C$ is an absolute constant. Therefore
\begin{equation}\label{gu23}
\big|\ln (m!)-(m+1/2)\ln(m/e)\big|\leq C^{(1)}\qquad\text{ for any }m\geq 1,
\end{equation}
where $C^{(1)}\geq 1$ is a absolute constant. Thus, for $n\geq 1$,
\begin{equation*}
\ln(n!)-\ln((3n)!)\leq -2n\ln(n/e)-3n\ln 3+2C^{(1)}
\end{equation*}
and
\begin{equation*}
\begin{split}
\sum_{j=1}^nm_j[\ln((3j+5)!)-\ln(j!)]&\leq \sum_{j=1}^nm_j[2C^{(1)}+(2j+5)\ln(j/e)+(3j+5.5)\ln((3j+5)/j)]\\
&\leq\sum_{j=1}^nm_j[C^{(2)}+(2j+5)\ln(j/e)+3j\ln 3],
\end{split}
\end{equation*}
where $C^{(2)}\geq 2C^{(1)}$ is an absolute constant. In view of \eqref{gu21} we have $n=\sum_{j=1}^n jm_j$, thus
\begin{equation}\label{gu25}
\begin{split}
H_{m_1,\ldots,m_n}&\leq \sum_{j=1}^nm_j[C'''+(2j+5)\ln(j/e)-2j\ln (n/e)+\ln(m/e)]\\
&\leq\sum_{j=1}^nm_j[C^{(3)}+7\ln(j/e)-(2j-1)\ln (n/j)],
\end{split}
\end{equation}
where $C^{(3)}\geq C^{(2)}$ is an absolute constant and $m:=m_1+\ldots+m_n\leq n$. It is easy to see that
\begin{equation*}
\sum_{j\in[1,2n/3]\cap \Z}m_j[C^{(3)}+7\ln(j/e)-(2j-1)\ln (n/j)]\leq \sum_{j\in[1,2n/3]\cap \Z}C^{(4)}m_j
\end{equation*}
for some constant $C^{(4)}\geq C^{(3)}$. Also
\begin{equation*}
\sum_{j\in (2n/3,n]\cap \Z}m_j[C^{(3)}+7\ln(j/e)-(2j-1)\ln (n/j)]\leq C^{(3)}+7\ln n,
\end{equation*}
since $\sum_{j\in (2n/3,n]\cap \Z}m_j\leq 1$ (due to the restriction \eqref{gu21}). Therefore it follows from \eqref{gu24} and \eqref{gu25} that
\begin{equation}\label{gu27}
\frac{\langle\alpha\rangle |D^n_\alpha L'(\alpha+i\beta,k)|(|k|\lambda_0^3)^{n+1}}{(3n)!}\leq n^7\sum_\ast \frac{[C_3^\ast]^{m_1+\ldots+m_n}}{m_1!\cdot\ldots\cdot m_n!}
\end{equation}
for any $n\geq 1$, where $C_3^\ast=C_3^\ast(d,\lambda_0,\vartheta)$ is a constant.

Let $C_4^\ast:=\sum_{m\geq 0}(C_3^\ast)^m/(m!)$, let $\rho\in(0,1/2]$ be such that $(C_4^\ast)^\rho=1.005$, and let $J$ denote the largest integer $\leq \rho n$. Notice that if the $n$-tuple $(m_1,\ldots,m_n)$ satisfies \eqref{gu21} then
\begin{equation*}
m_{J+1}+\ldots+m_n\leq 1/\rho.
\end{equation*}
In particular, the number of possible choices of the integers $m_{J+1},\ldots,m_n$ is $\lesssim n^{1/\rho}$. Thus
\begin{equation*}
\sum_\ast \frac{[C_3^\ast]^{m_1+\ldots+m_n}}{m_1!\cdot\ldots\cdot m_n!}\leq\sum_{m_1,\ldots,m_J\in\Z_+}\sum_{m_{J+1}+\ldots+m_n\leq 1/\rho}\frac{[C_3^\ast]^{m_1+\ldots+m_n}}{m_1!\cdot\ldots\cdot m_n!}\lesssim (C_4^\ast)^{J}n^{1/\rho}.
\end{equation*}
The desired bounds \eqref{pam25} follow from \eqref{gu27}.

(ii) We can complete the proof of the lemma. In view of the analyticity of $L'$, we can write
\begin{equation*}
\widehat{G}(s,k)=\frac{1}{2\pi}\int_\R L'(\alpha+i\beta,k)e^{i(\alpha+i\beta) s}\,d\alpha=\frac{e^{-\beta s}}{2\pi}\int_\R L'(\alpha+i\beta,k)e^{i\alpha s}\,d\alpha.
\end{equation*}
We use the uniform bounds \eqref{pam25} and let $\beta\to-\infty$ to conclude that $\widehat{G}(s,k)=0$ if $s<0$. 

To prove the main estimates \eqref{pam27.2} we may assume $|sk|\geq 1$, use the formula above with $\beta=0$, and integrate by parts $a\geq 1$ times when $|\alpha|\leq e^{s|k|}$. Therefore for any $s\geq 1/|k|$
\begin{equation*}
\begin{split}
2\pi|\widehat{G}(s,k)|&\leq \int_{|\alpha|\geq e^{s|k|}}|L'(\alpha,k)|\,d\alpha+|s|^{-a}\Big|\int_{|\alpha|\leq e^{s|k|}}e^{is\alpha}(D^a_\alpha L')(\alpha,k)\,d\alpha\Big|\\
&+\sum_{a'=0}^{a-1}|s|^{-a'-1}\big[|D^{a'}_\alpha L'(e^{s|k|},k)|+|D^{a'}_\alpha L'(-e^{s|k|},k)|\big].
\end{split}
\end{equation*}
We estimate the first term using the $L^2$ bounds \eqref{gu19.1} and the other two terms using the $L^\infty$ bounds \eqref{pam25}. Therefore for any $k\in\Z^d\setminus\{0\}$, $s\geq 1/|k|$, and $a\geq 1$ we have
\begin{equation*}
\begin{split}
|\widehat{G}(s,k)|\lesssim e^{-s|k|/2}+|sk||s|^{-a}\frac{(3a)!(1.01)^{a}}{(|k|\lambda_0^3)^{a+1}}+e^{-s|k|}\sum_{a'=0}^{a-1}|s|^{-a'-1}\frac{(3a')!(1.01)^{a'}}{(|k|\lambda_0^3)^{a'+1}}.
\end{split}
\end{equation*}
Therefore for any $k\in\Z^d\setminus\{0\}$, $s\geq 1/|k|$, and $a\geq 1$
\begin{equation}\label{pam30}
\begin{split}
|\widehat{G}(s,k)|\lesssim e^{-s|k|/2}+|sk|\frac{(3a)!(1.01)^{a}}{(s|k|\lambda_0^3)^{a}}+e^{-s|k|}\sum_{a'=0}^{a-1}\frac{(3a')!(1.01)^{a'}}{(s|k|\lambda_0^3)^{a'}}.
\end{split}
\end{equation}

We let now $a$ denote the smallest integer $\geq \lambda_0(s|k|)^{1/3}/3$. Using the bounds \eqref{gu23} below, 
\begin{equation*}
\begin{split}
\frac{(3a')!(1.01)^{a'}}{(s|k|\lambda_0^3)^{a'}}&\lesssim \exp\big[(3a'+1/2)(\ln (3a'+1)-1)-3a'\ln(\lambda_0(s|k|)^{1/3})+a'\ln(1.01)\big]\\
&\lesssim \exp\big[-3a'+(1/2)\ln(3a'+1)+a'\ln(1.01)\big]
\end{split}
\end{equation*} 
for any integer $a'\in[0,a]$. The desired bounds \eqref{pam27.2} follow from \eqref{pam30}.

\end{document}